\documentclass{article}

\let\amsamp=&

\usepackage{amsmath, amssymb, amsthm}
\usepackage[usenames,dvipsnames]{xcolor}
\usepackage{hyperref}
\usepackage{microtype} %% make the text look pretty
\usepackage{subcaption}
\usepackage[all]{xy}
\usepackage[capitalize,noabbrev]{cleveref}
\usepackage{mathdots}

\usepackage{tikz-cd,tikz, tikz-3dplot}
\usetikzlibrary{calc,positioning,}
\usepackage{pgfplots}
\usepgfplotslibrary{fillbetween}
\tikzset{
  partial ellipse/.style args={#1:#2:#3}{
    insert path={+ (#1:#3) arc (#1:#2:#3)}
  }
}
\begingroup
\catcode`\&=13
\gdef\pampmatrix{%
  \begingroup 
  \let&=\amsamp
  \setlength\arraycolsep{1pt} 
  \begin{bmatrix}%
}
\gdef\endpampmatrix{\end{bmatrix}\endgroup}
\endgroup

\usepackage[
backend=bibtex,
style=numeric,
block=space,
maxbibnames=99
]{biblatex}
\renewbibmacro{in:}{}
\DeclareNameAlias{sortname}{first-last}

\definecolor{lred}{RGB}{228,78,97}
\definecolor{lblue}{RGB}{0,179,230}
\definecolor{dblue}{RGB}{20,95,170}
\definecolor{cblue}{rgb}{0.39, 0.58, 0.93}
\definecolor{cgreen}{HTML}{43AA8B}
\usepackage[hmargin=2.85cm,vmargin=3.5cm]{geometry}
\newcommand{\Mdef}[2]{\newcommand{#1}{\relax \ifmmode #2 \else $#2$\fi}}
\Mdef{\F}{\mathbb{F}}
\Mdef{\M}{\mathcal{M}}
\Mdef{\N}{\mathbb{N}}
\Mdef{\R}{\mathbb{R}}
\Mdef{\Rop}{\mathbb{R}^{\mathrm{op}}}
\Mdef{\Z}{\mathbb{Z}}
\Mdef{\Q}{\mathbb{Q}}
\Mdef{\cC}{\mathcal{C}}
\Mdef{\cM}{\mathcal{M}}
\Mdef{\cW}{\mathcal{W}}
\Mdef{\cF}{\mathcal{F}}

\newcommand{\cat}[1]{\mathbf{#1}}
\newcommand{\IntI}{\cat{Int_{I}}}

\newcommand{\vect}{\cat{Vect_k}}
\newcommand{\uR}{\cat{R}}
\Mdef{\uD}{\mathbf{\Delta}^2}
\newcommand{\Top}{\cat{Top}}

\newcommand{\ra}{\rightarrow}

\newcommand{\incl}{\hookrightarrow}
\newcommand{\xto}{\xrightarrow}

\newcommand{\isom}{\cong}

\newtheorem{theorem}{Theorem}[section]
\newtheorem{proposition}[theorem]{Proposition}

\newtheorem{remark}[theorem]{Remark}

\newtheorem{definition}[theorem]{Definition}
\newtheorem{example}[theorem]{Example}

\DeclareMathOperator{\rank}{rank}

\title{Multiparameter persistent homology via generalized Morse theory}
\author{Peter Bubenik and Michael J. Catanzaro}
\date{}

\begin{document}
\maketitle

\begin{abstract}
  We define a class of multiparameter persistence modules that arise from a one-parameter family of functions on a topological space
  and prove that these persistence modules are stable.
  We show that this construction can produce indecomposable persistence modules with arbitrarily large dimension.
    In the case of smooth functions on a compact manifold, we apply cobordism theory and Cerf theory to study the resulting persistence modules.
  We give examples in which we obtain a complete description of the persistence module as a direct sum of indecomposable summands and provide a corresponding visualization.
\end{abstract}

%\tableofcontents

\section{Introduction}
% Topological Data Analysis (TDA) consists of a suite of tools combining algebraic topology and statistics to understand the `shape' of data. One of the most
% commonly used tools is known as persistent homology. Persistent homology has
% proven to be extremely successful in an array of applications ranging from 
% biological~\cite{chan_topology_2013, singh_topological_2008} to 
% physics-based~\cite{pranav_topology_2017,pranav_unexpected_2018} to even
% financial~\cite{gidea_topological_2018, gidea_topological_2017, truong_exploration_nodate}.
Persistent homology is an important tool in topological data analysis, whose goal is to use ideas from topology to understand the `shape of data'~\cite{ghrist:survey,carlsson:topologyAndData}. 

An important example of persistent homology starts with a smooth, compact manifold $M$ and a Morse function $f:M \to \R$.
Classical Morse theory concerns itself
with the study of $M$ via the critical points of $f$ by analyzing the sublevel
sets $F(a) = f^{-1}(\infty,a]$ and how their topology changes as $a$ varies. % through critical values.
The subspaces $\{F(a)\}_{a \in \R}$ and their inclusion maps may be used to define a functor $F: \cat{R} \to \cat{Top}$, where $\cat{R}$ is the category given by the linear order on $\R$ and $\cat{Top}$ is the category of topological spaces and continuous maps.
Composing with singular homology in some degree $j$ and coefficients in a field $k$ we obtain a functor $H_{j}F: \cat{R} \to \cat{Vect_k}$ 
with codomain the category of $k$-vector spaces and $k$-linear maps.
Such a functor is called a persistence module.

  Let $\beta_j$ denote the $j$-th Betti number of $M$ and let $M_j$ denote the number of critical points of index of $j$ of $f$. Let $M(t) = \sum_j M_j t^j$ and $\beta(t) = \sum_j \beta_j t^j$. Morse observed that $M(t)-\beta(t) = (1+t)D(t)$ for some polynomial $D(t)$ with non-negative coefficients~\cite{Morse:1925}. That is, the excess of critical points of the Morse function come in pairs that differ in index by one.
  A strengthening of this observation is a central result in persistent homology. The persistence module $H_{j} F$ decomposes into a direct sum of indecomposable summands given by one-dimensional vector spaces supported on an interval. The end points of these intervals are exactly the critical values of the paired critical points in Morse's theorem.
  This pairing of critical values, called the persistence diagram, is central to persistent homology.

% Persistent homology takes its cue from this classical story. 
% % except persistence
% % also keeps track of the critical values in addition to critical points. 
% The
% \emph{$i^\mathrm{th}$ persistent homology of $f$} is the assignment of the
% vector space $H_i(F(a);k)$ to each real number $a \in \R$, where $H_i(-;k)$
% denotes the $i$-th singular homology functor with coefficients in a field $k$,
% together with linear transformations $H_i(F(a);k) \to H_i(F(b);k)$, whenever $a
% \leq b$. It is natural to think of this assignment as a functor $\uR \to
% \vect$, where $\uR$ is the category given by the poset of real numbers and
% $\vect$ is the category of vector spaces over $k$. In most applications,
% we relax the requirements on $M$ to be a metric space and $f$ to be
% the distance function to a finite set of (data) points. The homotopy type of
% the sublevel sets $F(a)$ change for only finitely many values of $a$~\cite{cohen-steiner_stability_2007},
% so this framework follows classical Morse theory closely.

While this setting has been very successful, in many applications the data are
best described not by a single function $f:M \to \R$ but by a one-parameter
family of functions $f_t:M \to \R$, where $t \in I = [0,1]$. For example, one
may handle noise in the data with a procedure dependent on a parameter $t$.
The resulting homological data may be encoded in a multiparameter persistence
module.  However, in general this module does not decompose into one-dimensional summands,
and there is no complete invariant analogous to the persistence
diagram~\cite{carlsson_theory_2009}.

% Multiparameter persistence, on the other hand, is a term used for any
% generalization of ordinary persistence to the case when multiple parameters are
% varied simultaneously.  We focus ourselves with enlarging the domain of
% persistent homology from $\uR$ to $\uR \times \IntI$, where $\IntI$ is the
% category of sub-intervals of $I$ with morphisms given by inclusion of
% intervals.  We do so by replacing the Morse function $f$ above with a generic
% one-parameter family of smooth functions $\tilde f: I \times M \to \R$.  The
% goal of this paper is to introduce this new version of multiparameter
% persistent homology using families of smooth functions by taking sub-level sets
% as in classical Morse theory and simultaneously restricting a sub-interval of 
% parameter (`time') values.  We aim to complement
% the ever growing list of algebraic approaches to this topic by introducing the
% ideas of differential topology into this area~\cite{miller_data_2017, harrington_stratifying_2019, lesnick_computing_2019, carlsson_theory_2009}.
% We also hope our visualizations of multiparameter persistence modules can
% be used in unison with other visualization frameworks~\cite{lesnick_interactive_2015}.

We approach multiparameter persistent homology using two distinct generalizations of Morse theory. The
first is a parametrized approach to Morse theory, known as Cerf theory. Cerf
theory was initiated by J. Cerf in his celebrated proof of the
Pseudo-Isotopy Theorem~\cite{cer:str:1970}.  One outcome of his work was a useful
stratification on the space of all smooth functions on a smooth compact manifold, 
stratified by singularity type. The existence of this stratification implies that generic, 1-parameter
families of smooth functions are almost always Morse, except for finitely many
parameter values at which the function may have cubic, or `birth-death' type, singularities.
Cerf developed a convenient framework for understanding how singularities 
merge, split, and pass one another in families. This understanding is paramount
for our analysis.

The second variant is Morse theory adapted to the case of manifolds with
boundary. This was developed around the same time as Cerf theory but by several authors  
independently~\cite{hajduk_minimal_1981,braess_morse-theorie_1974,jankowski_functions_1972}.
Many statements in classical Morse theory can be adapted to manifolds with
boundary, so long as the gradient or gradient-like flow used is tangent to the boundary. 
Critical points occurring on the interior behave as one expects, but on the boundary
they come in two distinct flavors, either stable or unstable, depending on the
local flow. Altogether, Morse theory for manifolds with boundary is a 
powerful extension of its classical analog.
% We believe
We have only touched
the surface of using this subject in multiparameter persistence.

\subsection*{Our contributions}

%Our contribution is as follows.
Given a
% smooth, compact manifold
topological space
$X$ and a
one-parameter family of
% smooth
(not necessarily continuous)
functions $\tilde f:I\times X \to \R$, we define a
fibered version of $\tilde f$ by letting $f:I \times X \to I \times \R$ be given by
$f(t,x) = (t, \tilde f(t,x))$.
The collection of subspaces
\begin{equation} 
  F(a,b,c) = f^{-1}([a,b] \times (-\infty, c]) \subset I \times X \, ,
  \label{eqn:collection_of_spaces}
\end{equation}
for $0 \leq a \leq b \leq 1$ and $c \in \R$, are
our main objects of study.
For $[a,b] \subset [a',b']$ and $c \leq c'$ there is an inclusion of $F(a,b,c)$ into $F(a',b',c')$.
These topological spaces and continuous maps describe a functor
$F: \cat{Int_I \times R} \to \cat{Top}$, where $\cat{Int_I \times R}$ is the category given by the product of the partial order on the closed intervals in $I$ and the linear order on $\R$.
Composing with singular homology in some degree $j$ with coefficients in a field $k$, we obtain a functor $H_{j} F: \cat{Int_I \times R} \to \cat{Vect_k}$.
This functor is a multiparameter persistence module.
% In Section~\ref{sec:examples},

We prove that this functor is stable with respect to the interleaving distance for perturbations of the one-parameter family of smooth functions (Theorem~\ref{thm:stability}).

We consider several examples of such one-parameter families of functions and give complete descriptions of their multiparameter persistence modules
(Sections \ref{sec:pmod_func}, \ref{sec:indecomposable}, \ref{sec:cyl}, and \ref{sec:wrinkled-cylinder}).
In particular, we give decompositions of these modules into their indecomposable summands and provide corresponding visualizations
(Figures \ref{fig:pm-hat}, \ref{fig:pm-cylinder}, and \ref{fig:wrink_cyl_expl}).
We also show that indecomposable persistent modules arising from one-parameter families of functions may have arbitrarily large dimension (Section~\ref{sec:arb-dim}).

Now consider the case where $X$ is a smooth compact manifold and $\tilde{f}$ is smooth.
  Let $F_0(a,b,c) = f^{-1}([a,b] \times \{c\})$.
  We prove (\cref{prop:morse_proj}) that
for generic
$a$, $b$, and $c$,
%our \cref{prop:morse_proj} states that
$(F(a,b,c),F_0(a,b,c))$
forms a cobordism between the manifolds with boundary
$(F(a,a,c),F_0(a,a,c))$ and $(F(b,b,c),F_0(b,b,c)$.
Furthermore, it
is naturally equipped
with a Morse function
    \[ 
      \pi_{[a,b]}: F(a,b,c) \to [a,b] \, , 
    \]
given by projection onto the interval $[a,b]$.
This Morse function has no critical points on $F(a,b,c) \setminus F_0(a,b,c)$ and the positive and negative critical points on $F_0(a,b,c)$
(\cref{def:positive-negative}) correspond to boundary stable and boundary unstable critical points, respectively.

We remark that if we restrict
our collection of subspaces in \eqref{eqn:collection_of_spaces} to those with $a=0$ then we obtain a multiparameter persistence module indexed by $I \times \R \subset \R^2$, with the usual product partial order.
However, this two-parameter persistence module is a weaker invariant than our three-parameter persistence module.
  For example, if $X$ is the one-point space then our persistence module is a complete invariant of 1-parameter families of functions, but the weaker invariant is not.

We also remark that in applications computing the  full set of critical values (the Cerf diagram -- \cref{sec:cerf}) should not be considered to be a prerequisite. In the classical situation of sublevel set persistent homology of a single smooth function (e.g. a sum of a large number of Gaussians), instead of computing the set of critical values, one computes the sublevel set persistent homology of a piecewise linear approximation.

\subsection*{Related work}

  Much of the recent work on multiparameter persistent homology focuses on
  either its algebraic structure, for example~\cite{harrington_stratifying_2019,Bubenik:2021,Miller:2017,Bubenik:2018}, or its
  computational challenges, for
  example~\cite{lesnickWright:rivet,CFKLW:19,Kerber:2018a,Vipond:2020}.  
  Other authors have begun to take a more geometric approach similar to
  our own, such as~\cite{MacPherson:2021,grady_natural_2020}.  There are two recent
  geometric approaches that are related but still distinct.
  In~\cite{budney_bi-filtrations_2021}, the authors use handlebody
  theory to understand bi-filtrations arising from preimages of 2-Morse functions $f:M \to \R^2$.
  In a simliar vein, the authors of~\cite{k_biparametric_2021}
  use singularities of maps $M \to \R^2$ to understand preimages in $M$. 

\subsection*{Motivation}

  While our framework is theoretical, we are motivated by applications. 
  % We envision several applications of this work outside of differential topology and in particular, to 
  % statistical learning. 
  We highlight two examples: kernel density estimation and kernel regression.

\subsubsection*{Kernel Density Estimation}
  
Suppose $\{x_1, x_2, \ldots, x_n\} \subset \R^d$ are samples drawn
  independently from an unknown density function $f$.
  An empirical estimator of the density is obtained by the average of bump functions centered at each $x_i$. The bump functions are translations of a bump function, $K$, centered at the origin called a \emph{kernel}.
  That is,
  \begin{equation*}
    \hat{f}_{\alpha}(x) = \frac{1}{n\alpha} \sum_{i=1}^n K \left( \frac{x-x_i}{\alpha} \right),
    \end{equation*}
    where the parameter $\alpha$ is called the \emph{bandwidth}.
  % Kernel
  % density estimation approximates the unknown density function $f$ by centering a kernel
  % on each sample and taking an average. We focus on the \textit{Gaussian kernel density estimator $\hat f_{\alpha}$ of
  % bandwidth $\alpha >0$}, given by
  % For example, using the Gaussian kernel with bandwidth $\alpha > 0$, we have
  % \[
  %   \hat f_{\alpha}(x) = \frac{1}{n\alpha} \sum_{i=1}^n \frac{1}{\sqrt{2\pi}}\mathrm{exp}\left(\frac{-(x-x_i)^2}{2\alpha^2}\right) \,.
  % \]
    A standard choice is the Gaussian kernel, $K(x) = \frac{1}{(2\pi)^{d/2}} \exp(-{\lVert x \rVert}^2/2)$.
    % The function inside the sum is referred to as a kernel, and other kernels can be substituted
    Other examples include the Epanechnikov and triangular kernels, which appear (up to rescaling) as the functions $g(t)$ and $\tilde g(t)$,
    respectively, of Section~\ref{sec:pmod_func}. 

% \change{
%   To simplify the discussion, consider the simple problem of computing the number of modes of the unknown
%   distribution $f$. We use the estimator $\hat f_{\alpha}$ as a proxy for $f$ and instead compute its modes. 
%   However, counting modes depends on several choices for which there may not be natural answers. For example,
%   how does one choose the bandwidth $\alpha$? At what threshold does a local maximum of $\hat f_{\alpha}$ 
%   constitute a mode, i.e., how tall does a `peak' have to be before it counts toward a mode? Is it 
%   feasible to compute the local maxima?
% }

Properties of the kernel density estimator $\hat{f}$, such as the number of modes (i.e. local maxima), depend on the bandwidth $\alpha$.
  In order to obtain a global understanding of these properties for various of $\alpha$ and how they interact, 
  we consider the one-parameter family of functions $\tilde{g} = -\hat f: \R \times I \to \R$, where $\tilde{g}(t,x) = -\hat{f}_t(x)$ and $I$ is some bounded interval of parameter values. We obtain a collection of spaces, $G$, given by \eqref{eqn:collection_of_spaces} and its associated multiparameter persistence modules, $H_j G: \cat{Int_I} \times \cat{R} \to \cat{Vect_K}$.
  We may use $H_0G$ for a functorial analysis of the estimation of the modes of $f$. 
%
  % We construct a fibered version of $\hat f$ by letting
  % $f:I \times \R \to I \times \R$ be given by $f(\alpha,x) = (\alpha, \hat f_{\alpha})$.
  % We build a collection of spaces analogous to
  % that of Eq.~\eqref{eqn:collection_of_spaces} by setting
  % \[
  %   G(a,b,c) = f^{-1}([a,b] \times [c ,\infty)) \,.
  % \]
  % The multiparameter persistence module $H_0G(a,b,c)$ provides a detailed
  % analysis for mode estimation by freely varying the bandwidth $\alpha \in [a,b]$ and
%  mode threshold $c \in \R$.
  In particular, the dimension of
  $H_0G(\alpha,\alpha,-c)$ equals number of connected components of
  the superlevel set $f_{\alpha}^{-1}([c,\infty))$.  Furthermore, the linear maps $H_0G(a,a,-c) \rightarrow H_0G(a,b,-c) \leftarrow H_0G(b,b,-c)$
  allow one to study the persistence of these connected components.
  % yields a functorial description, so one can see how new modes are formed or
  % merged as the parameters $a$, $b$, and $c$ vary, rather than sampling the
  % number of modes at finitely many bandwidths and thresholds. In this sense,
  % the perspective provided by multiparameter persistence is novel and builds upon
  % other statistical approaches. \\

\subsubsection*{Kernel regression}

Closely related to kernel density estimation, is  kernel regression.
 Suppose we are given data $\{(x_1,y_1), \ldots, (x_n,y_n)\} \subset \R^d \times \R$ sampled from the graph of some unknown function $f:\R^d \to \R$. Consider the Nadaraya–Watson estimator
  \begin{equation*}
    \hat{f}_{\alpha}(x) = \frac{\sum_{i=1}^n K_{\alpha}(x-x_i) y_i}{\sum_{i=1}^n K_{\alpha}(x-x_i)}.
  \end{equation*}
  In the same way as for kernel density estimation, we obtain a one-parameter family of functions and associated persistence modules.
  
% \noindent\textbf{\change{Time-series analysis.}}

% \subsubsection*{\change{Time series analysis}}

% \change{
%   Suppose we have discrete time-series data from
% an unknown time-varying function $f_t(x)$.
% }

%\noindent\textbf{\change{Outline.}}

\subsection*{Outline}

  The paper is organized as follows. In Section~\ref{sec:prelim}, we recall
definitions from geometric topology and Cerf theory. We define our primary objects
of study including our multiparameter persistence modules in Section~\ref{sec:1-param}.
In Section~\ref{sec:examples}, we provide several examples of one-parameter
families of functions on manifolds, visualizations of the relevant cobordisms, and
analyze the multiparameter persistence modules. Finally in Section~\ref{sec:extra}, we
prove our main theoretical result that $F(a,b,c)$ is generically equipped with a Morse function
and analyze its critical points.

\section{Background}
\label{sec:prelim}
We start with providing some background from geometric topology.

\subsection{Manifolds with corners}

There are several different, inequivalent notions of manifolds with corners and
smooth maps between them in the differential topology
literature. The following is a brief summary of~\cite{joyce_manifolds_2012}.
Let $H^n_k = \{ (x_1, x_2, \ldots, x_n) \, | \, x_1, x_2, \ldots, x_k \geq 0\}$.
In particular, $H^n_0 = \R^n$ and $H^n_1 = [0,\infty) \times \R^{n-1}$.
\begin{definition}[{\cite[Definition 2.1]{joyce_manifolds_2012}}]
  Let $M$ be a second countable Hausdorff space. 
  \begin{itemize}
    \item An \emph{$n$-dimensional chart on $M$ without boundary} is a pair $(U,\psi)$, where
      $U$ is an open subset of $\R^n$ and $\psi: U \to M$ is a homeomorphism onto a nonempty open
      set $\psi(U)$.
    \item An \emph{$n$-dimensional chart on $M$ with boundary} for $n \geq 1$ is a pair $(U, \psi)$, 
      where $U$ is an open subset in $\R^n$ or $H^n_1$, and $\psi:U \to M$ is a homeomorphism
      onto a nonempty open set $\psi(U)$.
    \item An \emph{$n$-dimensional chart of $M$ with corners} for $n \geq 1$ is a pair $(U, \psi)$,
      where $U$ is an open subset of $H^n_k$ for $0 \leq k \leq n$, and $\psi: U \to M$ is a 
      homeomorphism onto a nonempty open subset $\psi(U)$.
  \end{itemize}
\end{definition}

\begin{definition}
  For $X \subset \R^n$ and $Y \subset \R^m$, a map $f : X \to Y$ is \emph{smooth} if it can be extended to
  a smooth map between open neighborhoods of $X$ and $Y$. If $m = n$ and $f^{-1}$ is also smooth, then
  $f$ is a \emph{diffeomorphism}.
\end{definition}

\begin{definition}
  An \emph{$n$-dimensional atlas for $M$ without boundary, with boundary, or with corners} 
  is a collection of $n$-dimensional charts without boundary, with boundary, or with corners
  $\{(U_j, \psi_j) \, | \, j \in J\}$ on $M$ such that $M = \cup_j \psi_j(M)$ and are
  compatible in the following sense: $\psi_j \circ \psi_k^{-1}: \psi_k^{-1}( \psi_j(U_j) \cap
  \psi_k(U_k)) \to \psi_j^{-1}(\psi_j(U_j) \cap \psi_k(U_k))$ is a
    diffeomorphism. An atlas is \emph{maximal} if it is not a proper subset of any
  other atlas.  
\end{definition}

\begin{definition}
  An \emph{$n$-dimensional manifold without boundary, with boundary, or with corners} is a second countable
  Hausdorff space $M$ together with a maximal $n$-dimensional atlas of charts without boundary, with boundary, or
  with corners.
\end{definition}

\begin{example}
  The space $\Omega$ of Figure~\ref{fig:cob_bdy} provides an example of a manifold with corners. 
  There are six corner points (with neighborhoods homeomorphic to $H_2^2$) at the intersections of $V_0$, $V_1$, and $Y$. The spaces
  $V_0$, $V_1$, and $Y$ are examples of manifolds with boundary. Their interiors, as well
  as the interior of $\Omega$, are examples of manifolds without boundary.
\end{example}

\subsection{Generalized Morse functions}

Morse theory provides powerful methods for understanding manifolds through the lens of smooth functions. 
Classical Morse theory concerns the study of smooth, compact manifolds without boundary
and allows for a transformation from smooth, continuous data (manifolds)
to discrete data (critical points and values). An adaptation to Morse theory for manifolds
with boundary extends this to the setting of cobordisms. Another generalization 
we will consider, known as Cerf theory, generalizes
this to the study of one-parameter families of functions. 
The remainder of this subsection is a summary and restatement of ideas 
from \cite[\S 1]{eli:wrinkling:1997} and \cite{mil:mor:1963}.

Let $M$ and $Q$ be smooth, compact manifolds of dimension $n$ and $q$, respectively, and let $f:M \to Q$ be a smooth map.
A point $p \in M$ is a {\em critical point} or {\em singular point}, if $\mathrm{rank}$ $d_pf=0$ or
\[
\mathrm{rank}\, d_pf < \min(n,q) \, .
\]
The set of all critical points of $f$ is denoted $\Sigma(f)$.

Assume $n \geq q$.
A point $p \in \Sigma(f)$ is a {\em fold singularity of index $j$} (see Figure~\ref{fig:fold-singularity}) if for some choice of local coordinates 
near $p$, the map $f$ is given by
\[
    \phi: \R^{q-1} \times \R^{n-q+1} \to  \R^{q-1} \times \R %\, \, ,
\]
\begin{equation}
    (t,x) \mapsto \left(t, -\sum_{i=1}^j x_i^2 + \sum_{i=j+1}^{n-q+1} x_i^2 \right) \,\, ,
  \label{eqn:fold_sing}
\end{equation}
where $t \in \R^{q-1}$ and $x = (x_1,x_2, \ldots, x_{n-q+1}) \in \R^{n-q+1}$. 
Let $\Sigma^{10}(f)$ be the set of all fold singularities.

For $q>1$, a point $p \in \Sigma(f)$ is a {\em cusp singularity of index
$j+1/2$} (see Figure~\ref{fig:cusp-singularity}) if for some choice of local coordinates near $p$, the map $f$ is given by
\[
    \psi: \R^{q-1} \times \R \times \R^{n-q} \to \R^{q-1} \times \R  %\,\, ,
\]
\begin{equation}
    (t,z,x) \mapsto \left(t, x^3 +3t_1z - \sum_{i=1}^j x_i^2 + \sum_{i=j+1}^{n-q}x_i^2\right) \,\, ,
\label{eqn:cusp_sing}
\end{equation}
where $t = (t_1,t_2, \ldots, t_{q-1}) \in \R^{q-1}$, $z \in \R$, and $x = (x_1,x_2,\ldots,x_{n-q}) \in \R^{n-q}$.
Set $\Sigma^{11}(f)$ to be the set of all
cusp singularities. Finally let $\Sigma^1(f) = \Sigma^{10}(f) \cup \Sigma^{11}(f)$.

\begin{remark}
	Consider the case $q=1$ and $Q \subset \R$, so that all terms of
	Eq.~\eqref{eqn:fold_sing} involving $t$ vanish.  In this case the fold
	singularities of $f:M \ra Q$ coincide with non-degenerate critical
	points as in usual Morse theory. If such an $f$ has only fold
	singularities, then $f$ is known as a Morse function.  
\end{remark}

\begin{remark} \label{rem:birth-death}
  Both fold and cusp singularities are locally fibered over $\R^{q-1}$ in the sense
  that the following commute
  \begin{equation*}
    \begin{minipage}{0.5\textwidth}
      \centering
      \begin{tikzcd}
          \R^{q-1} \times \R^{n-q+1}
	\arrow[r, "\phi"]
	\arrow[rd, "\pi"']
	&
    \R^{q-1} \times \R
	\arrow[d, "\pi"]
	\\
	& 
	\R^{q-1}
      \end{tikzcd}
    \end{minipage}
    \begin{minipage}{0.5\textwidth}
      \centering
      \begin{tikzcd}
          \R^{q-1} \times \R \times \R^{n-q} 
	\arrow[r, "\psi"]
	\arrow[rd, "\pi"']
	&
    \R^{q-1} \times	\R
	\arrow[d, "\pi"]
	\\
	& 
	\R^{q-1} \, ,
      \end{tikzcd}
    \end{minipage}
    \label{eq:fibered}
  \end{equation*}
  where $\pi$ is the projection onto $\R^{q-1}$.  A single fibered function can
  be interpreted as a family of functions $\phi_t:\R^{n-q+1} \ra \R$ or $\psi_t:
  \R \times \R^{n-q} \ra \R$, indexed over $t \in \R^{q-1}$. In this language,
  the folds are constant families (see Figure~\ref{fig:fold-singularity}).  The cusps consist of families of functions
  with two non-degenerate critical points of index $j$ and $j+1$ for $t_1 < 0$,
  no critical points for $t_1 > 0$, and a cubic or `birth-death' singularity of index $j + 1/2$
  for $t_1 = 0$ (see Figure~\ref{fig:cusp-singularity}).  
\end{remark}

\begin{figure}[h]
  \centering
  \begin{subfigure}[t]{0.45\textwidth}
    \centering
    \tdplotsetmaincoords{85}{120}
    \begin{tikzpicture}[tdplot_main_coords]
      \draw[domain=-1.8:0, thick, variable=\x,name path=LB] plot (\x,0,-\x*\x);
      \draw[domain=-1.8:0, thick, variable=\x,name path = RB] plot (\x,3,-\x*\x);
      \draw[domain=0:2, thick, variable=\x,name path=LF] plot (\x,0,-\x*\x);
      \draw[domain=0:2, thick, variable=\x,name path = RF] plot (\x,3,-\x*\x);
      \tikzfillbetween[of=LF and RF]{gray,opacity=0.2};
      \tikzfillbetween[of=LB and RB]{gray,opacity=0.2};
    \end{tikzpicture}
    \caption{}
    \label{fig:fold-singularity}
  \end{subfigure}
  \begin{subfigure}[t]{0.45\textwidth}
    \centering
    \tdplotsetmaincoords{85}{120}
    \begin{tikzpicture}[tdplot_main_coords]
      \draw[domain=-1.5:-0.511, thick, variable=\x, name path = ULB] plot (\x,-2,\x*\x*\x-0.8*\x);
      \draw[domain=-0.511:0.511, thick, variable=\x, name path = MLB] plot (\x,-2,\x*\x*\x-0.8*\x);
      \draw[domain=0.511:1.5, thick, variable=\x, name path = BLB] plot (\x,-2,\x*\x*\x-0.8*\x);
      \draw[domain=-1.27:0,densely dotted, variable=\x, name path = UMB] plot (\x,0,\x*\x*\x);
      \draw[domain=0:1.31, densely dotted, variable=\x, name path = LMB] plot (\x,0,\x*\x*\x);
      \draw[domain=-1.11:0, thick, variable=\x, name path = URF] plot (\x,2,\x*\x*\x+0.5*\x);
      \draw[domain=0:1.2, thick, variable=\x, name path = LRF] plot (\x,2,\x*\x*\x+0.5*\x);
      \tikzfillbetween[of=ULB and UMB]{gray,opacity=0.2};
      \tikzfillbetween[of=BLB and LMB]{gray,opacity=0.2};
      \tikzfillbetween[of=LRF and LMB]{gray,opacity=0.2};
      \tikzfillbetween[of=URF and UMB]{gray,opacity=0.2};
    \end{tikzpicture}
    \caption{}
    \label{fig:cusp-singularity}
  \end{subfigure}
  \caption{
  Local models for (a) a fold singularity and (b) a cusp singularity. The
  middle slice of the cusp singularity has a cubic singularity. This 
  is often referred to as a `birth-death' singularity, since 
  the two critical points to the left can be viewed as being ``born'' (moving right to left)
  or as ``dying'' (moving left to right).}
  \label{fig:fold_bd}
\end{figure}
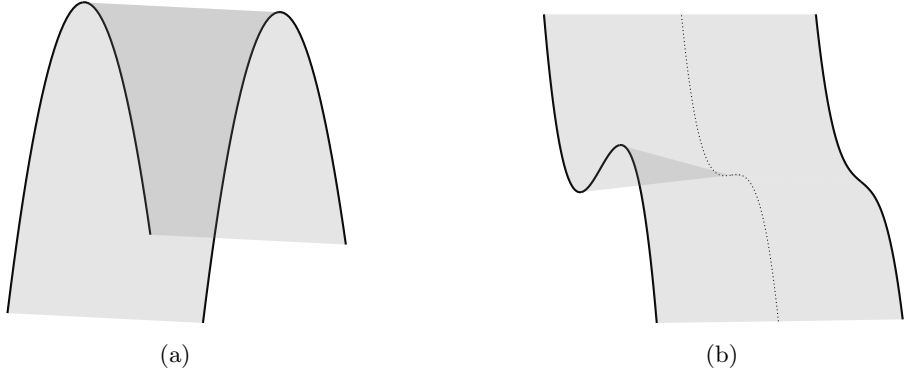

\subsection{Cobordisms}

In this section, we recall some basic notions from Morse theory
for manifolds with boundary from the well-written summary of~\cite{bor:mor:2016}. The reader may also consult
some of the original sources, such as~\cite{jan:fun:1972,braess_morse-theorie_1974, hajduk_minimal_1981}.
We will assume that all manifolds (with or without boundaries or corners) are smooth.

Let $V_0$ and $V_1$ denote two compact $n$-manifolds with boundaries $\partial V_0$
and $\partial V_1$, respectively. Let $\Omega$ be a compact $(n+1)$-manifold with corners,
$\partial \Omega = Y \cup V_0 \cup V_1$, where
% $Y$ is non-empty
% \pb{We should drop this condition since if $V_0$ and $V_1$ have empty boundary then we can have empty $Y$.}
$V_0 \cap V_1 = \varnothing$, and $Y \cap V_0 = \partial V_0$, $Y \cap V_1 = \partial V_1$.
In this case, we say $(\Omega,Y)$ is a {\em cobordism} between $(V_0,\partial V_{0})$
and $(V_1,\partial V_1)$.
See Figure~\ref{fig:cob_bdy}.
%
% \mc{Strictly speaking, $\Omega$ is a manifold with corners, so it locally looks like
% $\R^{n-2} \times \R^2_{\geq 0}$. Do we need to mention this?}
%
Such a cobordism is a {\em left-product cobordism}
if $\Omega$ is diffeomorphic to $V_0 \times [0,1]$, 
or a {\em right-product cobordism} if
$\Omega$ is diffeomorphic to $V_1 \times [0,1]$.

\begin{figure}[h]
  \centering
  \tdplotsetmaincoords{80}{120}
  \begin{tikzpicture}[tdplot_main_coords,scale=0.75]
    %% Exploded V_0
    \draw[domain=-2:1.8, thick, variable=\x,name path=LB,lblue] plot (\x-1.5,-4,\x*\x+0.5);
    \node[left] at (-2,-5,1) {$V_0$};
    %% Exploded V_1
    \draw[domain=-1.5:1.3, thick, variable=\x,name path = RB,cgreen] plot (\x-2,14.5,\x*\x+4.3);
    \draw[domain=-1.5:1.3, thick, variable=\x,name path = RF,cgreen] plot (\x+2,14.5,\x*\x+1.3);
    \node[right] at (0,15.5,3.5) {$V_1$};
    %% Exploded Y
    % Back line
    \draw[thick,lred] (-2,0,9) -- (-3.5,12,9);
    % Front line
    \draw[thick,lred] (1.8,0,8.24) -- (3.3,12,8);
    %Parabola
    \draw[thick, variable=\y,domain=-1:1,lred] plot(\y,{(2*\y*\y+10.5},8.5);
    \node[above] at (0,13,10) {$Y$};
    %%%%% Draw Omega
    \draw (0,0,1) -- (2,12,1); % bottom edge 
    \draw (0,9,4.8) -- (0,5,3.3); % top of middle tent
    \draw (0,10,4) -- (-2,12,4); % bottom of tent
    \draw[dotted] (0,10,4) -- (0,5,3.3);
    % top line in Y
    \draw[thick,lred]  (-2,0,5) -- (-3.5,12,6.25);
    % bottom line in Y
    \draw[thick,lred] (1.8,0,4.24) -- (3.3,12,2.69);
    % parabola in Y
    \draw[thick,lred] (-0.7,12,5.69) .. controls (0,8,5) .. (0.5,12,3.25);
    % Front of V_0
    \draw[domain=0:1.8, thick, variable=\x,name path = RF,lblue] plot (\x,0,\x*\x+1);
    % Back of V_0
    \draw[domain=-2:0, thick, variable=\x,name path = RB,lblue] plot (\x,0,\x*\x+1);
    % Front of V_1
    %bottom
    \draw[domain=0:1.3, thick, variable=\x,name path = RF,cgreen] plot (\x+2,12,\x*\x+1);
    %top
    \draw[domain=0:1.3, thick, variable=\x,name path = RF,cgreen] plot (\x-2,12,\x*\x+4);
    % Back of V_1
    \draw[domain=-1.5:0, thick, variable=\x,name path = RB,cgreen] plot (\x+2,12,\x*\x+1);
    \draw[domain=-1.5:0, thick, variable=\x,name path = RB,cgreen] plot (\x-2,12,\x*\x+4);
    \node[below] at (0,2,0.8) {$\Omega$};
    %corners
    \end{tikzpicture}
    \caption{ A manifold with corners $\Omega$ provides a cobordism
      between two manifolds with boundary $V_0$ and $V_1$. The
    boundary $\partial V_0$ consists of two points and $\partial V_1$ consists of
    four points. The manifold with boundary $Y$ can be viewed as a cobordism
    between $\partial V_0$ and $\partial V_1$. Furthermore, $\Omega$ is a left-product
    cobordism.} 
    \label{fig:cob_bdy}
  \end{figure}

  Fixing a Riemannian metric on $\Omega$ allows us to consider the gradient
  $\nabla F$ of a smooth function $F:\Omega \ra [a,b]$.  A critical point $z$ of
  $F$ is {\em Morse} if the Hessian of $F$ at $z$ is non-degenerate.  The
  function $F$ is a {\em Morse function} on the cobordism $(\Omega, Y)$ if
  % $F(V_0) = 0$, $F(V_1) = 1$,
  $F^{-1}(a) = V_0$, $F^{-1}(b) = V_1$,
  there are no critical points on $V_0 \cup V_1$, $F$
  only has Morse critical points, and $\nabla F$ is everywhere tangent to $Y$.

  The {\em unstable manifold} $W^u_z$ of
  a critical point $z$ is the set of all points which flow out from $z$ under $\nabla F$:
  \[
  W^u_z =  \{ x \, | \, \lim_{t \to -\infty} \Phi_t(x) = z\} \, ,
  \] 
  where $\Phi_t$ is the flow generated by $\nabla F$. With the same notation, the {\em stable
  manifold} $W^s_z$ of a critical point $z$ is given by
  \[
  W^s_z = \{ x \, | \, \lim_{t \to \infty} \Phi_t(x) = z\} \, .
  \]
  The stable and unstable manifolds are locally embedded disks~\cite{irwin_stable_1970}.

  Unlike usual Morse theory, the critical points for a Morse function on a
  manifold with boundary come in a variety of types.  If $z$ is a critical point
  and $z \in Y$, then $z$ is called a \emph{boundary critical point}.  Otherwise,
  $z$ is called an \emph{interior critical point}.  We are primarily interested in boundary critical
  points, of which there are again two types, determined by the gradient flow. A boundary
  critical point is {\em boundary stable} if $T_z W^{u}_z \subset T_z Y$; otherwise it is 
  {\em boundary unstable}.

  As usual, the {\em index} of a boundary critical point $z$ is defined as the dimension
  of the stable manifold $W^s_z$.
  If $z$ is boundary stable, then the index of $z$ is the usual index of $F|_Y$ plus one.
  On the other hand,
  if $z$ is boundary unstable, then the index of $z$ coincides
  with the usual notion of index of the restriction $F|_Y$.
  See Example~\ref{ex:Omega}.

  \begin{remark}
    Note that there are no boundary unstable critical points of index $n+1$, or boundary
    stable critical points of index $0$.
  \end{remark}

  \begin{remark}
	  \label{rem:flow_swap}
      We consider the flow generated by $\nabla F$, as is frequently used in most mathematics
      literature. In other areas such as dynamical systems and physics, the flow
      generated by $-\nabla F$ is commonly used. The two versions are equivalent, since the
      stable and unstable manifolds swap after replacing the flow generated by $\nabla F$
      with that generated by $-\nabla F$.
  \end{remark}

  \begin{proposition}[{\cite[Lem 2.10, Thm 2.27, Prop 2.38]{bor:mor:2016}}]
    \label{prop:morse-thy-bdy-cob}
    Let $(\Omega,Y)$ be a cobordism between $(V_0,\partial V_0)$ and $(V_1,\partial V_1)$.
    \begin{itemize}
      \item If $(\Omega,Y)$ admits a Morse function whose critical points are all 
	boundary stable, then $(\Omega,Y)$ is a left-product cobordism. 
      \item If $(\Omega,Y)$ admits a Morse function whose critical points are all boundary
	unstable, then $(\Omega,Y)$ is a right-product cobordism.
      \item If $(\Omega,Y)$ admits a Morse function with no critical points,
	      then $(\Omega,Y)$ is both a left- and right-product cobordism.
    \end{itemize}
  \end{proposition}

  \begin{example} \label{ex:Omega}
    In Fig.~\ref{fig:cob_bdy}, projection of $\Omega$ onto the horizontal axis yields a 
    Morse function $F: \Omega \ra [0,1]$, in which
%    $F(V_0) = 0$ and $F(V_1)=1$.
  $F^{-1}(0) = V_0$, $F^{-1}(1) = V_1$.
    This function has no interior critical points and a single boundary critical
    point. The boundary critical point is boundary stable, and located at the
    vertex of the parabola of $Y$ in $\Omega$. This is an index 1 critical point.
    Proposition~\ref{prop:morse-thy-bdy-cob} implies $\Omega$ is a left-product
    cobordism, as is evident from Figure~\ref{fig:cob_bdy}.

    If we post-compose $F$ with the involution $t \mapsto 1-t$, then we again have a
    Morse function with no interior critical points. This composition has the same boundary
    critical point as before but now it is boundary unstable. The
    index of this critical point is $1$.
  \end{example}

  \subsection{Cerf theory} \label{sec:cerf}
  
  Let $X$ be a smooth, compact $n$-manifold and let $I = [0,1]$ denote the unit interval. 
  A one-parameter family of functions on $X$ is a family of smooth functions $\tilde f_t: X \ra \R$,
  where $t \in I$, and the family varies smoothly with respect to $t$. 
  This is equivalent to specifying a single smooth function $\tilde f: I \times X \ra \R$. In 
  either case, this data gives rise to a map fibered over the interval
  \[
  f: I \times X \ra I \times \R \, , \quad f(t,z) = (t, \tilde f(t,z)) \, ,
  \]
  in the sense that the following diagram commutes
  \begin{equation}
    \begin{tikzcd}
      I \times X 
      \arrow[r, "f"]
      \arrow[rd, "\pi_I"']
      &
      I \times \R
      \arrow[d, "\pi_I"]
      \\
      & 
      I
    \end{tikzcd}
  \end{equation}
  where $\pi_I$ is projection onto the $I$ factor.

  Our primary tool for understanding such families of functions is the
  Cerf diagram. 

  \begin{definition} \label{defn:Cerf-diagram}
	  The {\em Cerf diagram} (or \emph{Kirby diagram}) 
	  of a family of functions $\tilde f:I \times X \to \R$
    is given by
    \[
      \bigcup_{t \in I, \, x \in \Sigma(\tilde f_t)} (t, \tilde f_t(x))
      \subset I \times \R \, .
    \]
    We label each %fold singularity
    nondegenerate critical value
    of $\tilde{f}_{t}$ with its corresponding index.
  \end{definition}

  The Cerf diagram encodes the critical value
  information of a family of functions as the time parameter
  $t$ varies~\cite{cer:str:1970,hat:pse:1973,kir:calc:1978}.
  A simple Cerf diagram is shown in Figure~\ref{fig:cerf}. Each
  (non-end) point on the curves corresponds
  to a
  % fold singularity,
  nondegenerate critical value of $\tilde{f}_t$
  and the points where two such curves terminate 
  is a cubic
  % birth-death 
  singularity of $\tilde{f}_{t}$.

  \begin{figure}[h]
    \centering
    \begin{tikzpicture}[xscale=0.5,yscale=0.3]
      \draw (0,-5) -- (0,5);
      \draw (10,-5) -- (10,5);
      \draw plot [smooth,tension=1.2] coordinates {(0,2) (5,4) (10,1.8)};
      \node [below right,scale=0.65] at (0,2) {2};
      \draw plot [smooth,tension=1.3] coordinates {(0,3.5) (4,1) (6.5,-0.5)};
      \node [above right,scale=0.65] at (0,3.5) {2};
      \draw plot [smooth,tension=1.2] coordinates {(2.5,-1) (6,2) (10,4)};
      \node [above,scale=0.65] at (2.6,-0.9) {2};
      \draw plot [smooth,tension=1.3] coordinates {(2.5,-1) (5,-3) (7,-4)};
      \draw plot [smooth] coordinates {(1.5,-3.5) (5,-1)  (6.5,-0.5)};
      \draw plot [smooth] coordinates {(1.5,-3.5) (7,-4)};
      \node [above,scale=0.65] at (1.6,-3.3) {1};
      \node [below right,scale=0.65] at  (1.5,-3.4) {0};
      \node[above,scale=0.65] at (6.8,-4) {1};
      \draw plot [smooth, tension=1.1] coordinates {(5,-0.4) (6.3, 0.8) (7.4,-0.4)};
      \draw plot [smooth, tension=1.1] coordinates {(5,-0.4) (6.3, -1.5) (7.4,-0.4)};
      \node [above,scale=0.65] at (7.4,-0.3) {1};
      \node [below,scale=0.65] at (7.4,-0.5) {0};
    \end{tikzpicture}
    \caption{ A Cerf diagram for a certain generic family of smooth functions. Each %fold singularity
      nondegenerate critical value
      is labeled by the index of its critical point.
      % \pb{The bottom part of this diagram is a dovetail -- Cerf p93 and Hatcher Wagoner p36, 102, 147, 169. Better yet, see Igusa, ``The space of framed functions'' Section A.3, and in particular Figures E,F,G. I especially like the comment ``It looks like a swallowtail to me!’’.}
See also \cite[Section A.3, Figures E, F, G.]{Igusa:1987}.
    }
    \label{fig:cerf}
  \end{figure}
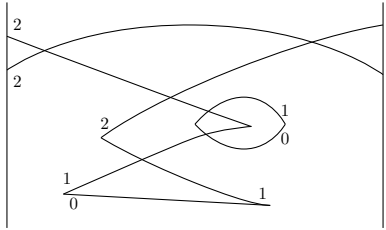

  We will assume that the family %$\tilde{f}_t$
  $\tilde{f}$ is
  \textit{generic}, meaning that for all but finitely many $t$, the fiber
  $\tilde{f}_t$ has finitely many %fold singularities
  nondegenerate critical points each of which has a distinct critical value.
  Furthermore, we
  will assume that all remaining fibers have %singularities consisting of
  either finitely many %fold singularities
  nondegenerate critical points
  exactly two of which have a common critical value or a single cubic
  % birth-death
  singularity and finitely many
  % fold singularities
  nondegenerate critical points all of which have distinct critical values.

  \subsection{Wrinkled maps}
  \label{sec:wrinkle}

  We recall the notion of a wrinkle from~\cite{eli:wrinkling:1997}.
  Let
  \[
  w:\R^{q-1} \times \R^{n-1} \times \R^1 \ra \R^{q-1} \times \R^1
  \]
  be given by
  \[
  w(t,x,z) = \left(t, z^3 + 3\left(|t|^2-1\right)z - \sum_{i=1}^j x_i^2 + 
  \sum_{i=j+1}^{n-q} x_i^2 \right) \, ,
  \]
  where $|t|^2 = \sum_{i=1}^{q-1}t_i^2$. The set of critical points of $w$ is 
  \[
  \Sigma(w) = \{ x = 0 \,,\,z^2 + |t|^2 = 1\} \subset \R^{q-1} \times \R^{n-q} \times \R \,\,,
  \]
  and can be identified with a $(q-1)$ sphere $S^{q-1} \subset \R^{q-1} \times \{0\} \times \R$. This sphere has
  equator
  \[
  \{x = 0 \,,\, z =0 \,,\, |t|=1\} \subset \Sigma^1(w) \, ,
  \]
  which we identify with $S^{q-2}$. This equator consists of cusp singularities of index $j+1/2$, its
  upper hemisphere $\Sigma^1(w) \cap \{z > 0 \}$ consists of fold singularities of index $j$,
  and the lower hemisphere $\Sigma^{1}(w) \cap \{ z < 0\}$ consists of fold singularities of index $j+1$.
  %Let $D \cong D^q$ be the $q$-dimensional disc bounded by $\Sigma^1(w)$, given by
  %$\{x = 0 \,,\, z^2 + |t|^2 \leq 1 \}$.

  \begin{definition}[\cite{eli:wrinkling:1997}]
    \label{defn:wrinkled_map}
    For an open subset $U \subset M$, a map 
    $f: U \ra Q$ is called a {\em wrinkle of index $s+1/2$} if it 
    is equivalent to the restriction of $w$ to an open neighborhood 
    $V \supset D$, where $D$ is the $q$-dimensional disc 
    $\{z^2 + |y|^2 \leq 1, x=0\}$ bounded by $\Sigma^1(w)$.

    A map $f: M \ra Q$ is called {\em wrinkled}
    if there exist disjoint open subsets
    $U_1, U_2, \ldots , U_l \subset M$ such that
    \begin{itemize}
      \item for each $i$, $f|_{U_i}$ is a wrinkle, and 
      \item if $U = \cup_1^l U_i$, then $f|_{M \setminus U}$ is a 
	submersion.
    \end{itemize}
  \end{definition}

  \begin{definition}\label{defn:wrinkled_with_folds}
    A map $f:M \ra Q$ is called {\em wrinkled with folds} 
    if there exist disjoint open subsets
    $U_1, U_2, \ldots , U_l \subset M$ such that
    \begin{itemize}
      \item for each $i$, $f|_{U_i}$ is a wrinkle, and 
      \item if $U = \cup_1^l U_i$, then $f|_{M \setminus U}$ has only
	fold singularities.
    \end{itemize}
  \end{definition}

  The singular locus of a wrinkled map decomposes into a union of 
  wrinkles $S_i = \Sigma^1(f|_{U_i}) \subset U_i$. As before,
  each $S_i$ has a $(q-2)$-dimensional equator of cusps which 
  divides $S_i$ into 2 hemispheres of folds of adjacent indices.
  The singular locus of a wrinkled map with folds decomposes into a union of wrinkles and folds.

%%% Local Variables:
%%% mode: latex
%%% TeX-master: "1-parameter"
%%% End:

\section{Persistence modules for 1-parameter families of functions}
\label{sec:1-param}
In this section we define multiparameter persistence modules for $1$-parameter families of functions.
The unit interval $[0,1]$ is denoted by $I$.

\subsection{Indexing categories}
\label{sec:index-cat}

Let $\IntI$ denote the category whose objects are closed intervals $[a,b]
\subset I$, and whose morphisms $[a,b] \ra [c,d]$ are inclusions
$[a,b] \subset [c,d]$.  Let $\Delta^2 = \{(a,b) \, | \, 0 \leq a \leq b \leq
1\}$.  The category $\IntI$ is isomorphic to the category $\uD$, whose objects
are points $(a,b) \in \Delta^2$ and has a unique morphism $(a,b) \ra (c,d)$ if
and only if $c \leq a \leq b \leq d$.
Finally, let $\uR$ denote the category corresponding to the poset of real numbers
$(\R,\leq)$.  Then we have the isomorphic product categories $\IntI \times \uR$
and $\uD \times \uR$.

Note that there may not exist a map between two objects of $\uD \times \uR$,
in contrast to the (ordinary) sublevel-set persistence of
  Morse functions.
There does exist, however, a zig-zag of maps between any two objects due to
the fact that $\IntI \cong \uD$ is a join-semilattice.
In particular, $[a,b] \subset [\min(a,a'),\max(b,b')] \supset [a',b']$;
for example, see the two arrows in the third triangular slice in Figure~\ref{fig:pm-hat}.

For $n \geq 1$, let $(\R^n,\leq)$ be the set $\R^n$ together with the product
partial order. That is $(x_1,\ldots,x_n) \leq (y_1,\ldots,y_n)$ if and
only if $x_k \leq y_k$ for all $1 \leq k \leq n$.  Then the poset $\Delta^2$ includes in the
poset $(\R^2,\leq)$ under the mapping $(a,b) \mapsto (-a,b)$.  It follows that
the product poset $\Delta^2 \times \R$ includes in the poset $(\R^3,\leq)$ under
the mapping $(a,b,c) \mapsto (-a,b,c)$.  Thus we have an inclusion of
categories $\uD \times \uR \incl \uR^3$ where $\uR^3$ denotes the category
corresponding to the poset $(\R^3,\leq)$.
For a poset $P$ and $p \in P$, let $U_p = \{q \in P \ | \ p \leq q\}$, called the \emph{up-set} of $p$.
Then our persistence modules
may also be considered to be $\R^3$-graded modules over the monoid ring
$K[U_0]$, where $U_0$ is the up-set of $0 \in \R^3$
(see~\cite{Bubenik:2021,Miller:2017}).

\subsection{Diagrams of spaces} \label{sec:diagrams}

Let $\Top$ denote the category of topological spaces and continuous maps.  Let
$X$ be a topological space and let $\tilde{f}:I \times X \to \R$ be a
(not necessarily continuous in either variable) real-valued function on $I \times X$, which corresponds to a
one-parameter family of real-valued functions on $X$, 
given by $\tilde{f}_t(x) = \tilde{f}(t,x)$.
Let $f:I \times X \to I
\times \R$ be the function given by $f(t,z)=(t,\tilde{f}(t,z))$.  Then we have a
\emph{diagram of spaces} of $X$ given by $F:\IntI \times \uR \to \Top$ or
equivalently $F:\uD \times \uR \to \Top$ given by $F([a,b],c) =
f^{-1}([a,b]\times(-\infty,c])$ or $F(a,b,c) =
f^{-1}([a,b]\times(-\infty,c])$, and morphisms given by inclusions of the
corresponding inverse images. For any subcategory $\mathcal C$ of $\IntI \times \uR$,
we can restrict a diagrams of space $F$ to $\mathcal{C}$, forming a \emph{sub-diagram
of spaces indexed on $\mathcal{C}$}; we omit $\mathcal{C}$ if it is clear from context. 
If the subcategory is finite, we say the diagram of spaces if \emph{finite}.

% \begin{remark}
%   Although we are only interested in smooth maps presently, the constructions in this section,
%   % \ref{sec:diagrams}, 
%   Section \ref{sec:mppm}, and Section \ref{sec:euler} do not require the one-parameter family of functions 
%   \change{$\tilde{f}(t,x)$} to be continuous \change{in either variable}.
% \end{remark}

\begin{remark}
  The target category of a diagram of spaces of $X$ can be restricted to $\cat{Sub}(I \times X)$, the category
  whose objects are subspaces of $I \times X$ and whose morphisms are given by inclusion.
\end{remark}

\subsection{Multiparameter persistence modules} \label{sec:mppm}

% \pb{In section 4, we use $k$ for a field, so I changed $K$ to $k$ here.}

Let $\vect$ denote the category of vector spaces over a field $k$ and $k$-linear maps.

Given a one-parameter family $\tilde{f}$ of real-valued functions on a topological space $X$ as in \cref{sec:diagrams}, we have the corresponding diagram of topological spaces $F$.
For $j \geq 0$, let $H_j = H_j(-;k)$ denote the singular homology functor in degree $j$ with coefficients in the field $k$.
The \emph{multiparameter persistence module} corresponding to $\tilde{f}$ is
given by the functor $H_j F: \IntI \times \uR \to \vect$ or equivalently $H_j F:
\uD \times \uR \to \vect$.
% In \cref{sec:examples}, we will use the two points of view to give two
% different visualizations of multiparameter persistence modules.

% \mc{It seems like the Cerf diagram approach is easier to prove theorems, like when
%   the cobordism is/isn't a product, whereas the $\Delta^2$ approach is easier to 
% describe the irreducible pieces with.}

\subsection{Betti and Euler characteristic functions} \label{sec:euler}

For applied mathematicians, it is sometimes preferable to ignore persistence entirely (i.e. the morphisms in the persistence module) and only compute the pointwise Betti numbers, or cruder still, the pointwise Euler characteristic. While much of the mathematical structure is lost, being able to complete computations on vastly larger data sets may be more important.
  In this section we show how these coarser invariants fit within our framework.

Whenever they are well defined, we have the following.  For $j \geq 0$, 
the \emph{$j$-th Betti function} 
$\beta_j: \Delta^2 \times \R \to \Z$ is given by 
\[
  \beta_j(a,b,c) = \rank(H_j F(a,b,c)).
\]
The \emph{Euler characteristic function} $\chi: \Delta^2 \times \R \to \Z$ is given by
\[
  \chi(a,b,c) = \sum_j (-1)^j \beta_j(a,b,c).
\]
% We will see in~\cref{sec:examples} that these functions contain strictly less information than
% the corresponding persistence module, in contrast with their one-parameter persistent
% homology counterparts~\cite{carlsson_theory_2009}.
%
In cases where $F$ is given by a cellular complex, the Euler characteristic
equals the alternating sum of the number of cells of a given dimension.
% While this invariant is much coarser than the persistence module, there may be
% situations in which it is sufficient to quickly compute this invariant.

\subsection{Stability}
\label{sec:stability}

We prove that our multiparameter persistence modules are stable with respect to perturbations of the underlying one-parameter family of functions.

Let $X$ be a topological space and consider two one-parameter families of (not necessarily continuous) functions, $\tilde{f},\tilde{g}: I \times X \to \R$.
Let $F,G: \cat{\Delta^2} \times \cat{R} \to \cat{Top}$ be the corresponding diagrams of spaces defined in Section~\ref{sec:diagrams} and for $j \geq 0$,
let $H_j F,H_j G: \cat{\Delta^2} \times \cat{R} \to \cat{Vect_K}$ be the corresponding multiparameter persistence modules defined in Section~\ref{sec:mppm}.
Let $d_{\infty}(\tilde{f},\tilde{g}) = \sup_{(t,x) \in I \times X} \lvert \tilde{f}_t(x) - \tilde{g}_t(x) \rvert$.

We define a \emph{superlinear family of translations}~\cite[Section 3.5]{bdss:1} on $\cat{\Delta^2} \times \cat{R}$ given by $\Omega_{\varepsilon}(a,b,c) = (a,b,c+\varepsilon)$ for $\varepsilon \geq 0$.
The corresponding \emph{interleaving distance}~\cite[Definition 3.20]{bdss:1}, $d_I$, is given by the infimum of all $\varepsilon$ for which two diagrams or persistence modules indexed by $\cat{\Delta^2} \times \cat{R}$ are $\Omega_{\varepsilon}$-interleaved~\cite[Definitions 3.4 and 3.5]{bdss:1}.

\begin{theorem} \label{thm:stability}
  $d_I(H_j F,H_j G) \leq d_{\infty}(\tilde{f},\tilde{g})$.
\end{theorem}

\begin{proof}
  Let $\varepsilon = \sup_{(t,x) \in I \times X} \lvert \tilde{f}_t(x) - \tilde{g}_t(x) \rvert$.
  It follows from the definitions that $F$ and $G$ are $\Omega_{\varepsilon}$-interleaved.
  By \cite[Theorem 3.23]{bdss:1}, $H_j F$ and $H_j G$ are also $\Omega_{\varepsilon}$-interleaved.
\end{proof}

%%%%%%%%%%%%%%%%%%%%%%%%%%%%%%%%%%%%%%%%%%%%%%%%
%% EXAMPLES
%%%%%%%%%%%%%%%%%%%%%%%%%%%%%%%%%%%%%%%%%%%%%%%%

%%% Local Variables:
%%% mode: latex
%%% TeX-master: t
%%% End:

\section{Examples I: basic examples}
\label{sec:examples}
In this Section, we illustrate our results with several examples.  Recall that
a one-parameter family of functions on $X$ is
a function $\tilde f: I \times X \ra \R$
%is equivalent to a family
consisting
of functions
$\tilde f_t:X \to \R$, indexed by $t \in I$.
% Either perspective
This
gives rise to a
map fibered over the interval I,
\begin{equation}
  f: I \times X \to I \times \R \, , \quad f(t,z) = (t, \tilde f(t,z)) \, .
  \label{eqn:family_of_func}
\end{equation}
%
% \begin{remark}
%   \label{rem:pl_approx}
We will
  replace smooth functions by piecewise linear approximations to make the associated multiparameter persistence module easier to describe -- for an example, see Figure~\ref{fig:pmod_func}.
  This replacement does not affect the qualitative structure of the module but does change the support of its indecomposable summands.
%\end{remark}

\subsection{Persistence modules of graphs of functions}
\label{sec:pmod_func}

We begin by considering a one point space $X = \{\ast\}$. A
one-parameter family of functions $\tilde f_t:X \to \R$ is equivalent to a function $g: I \ra \R$,
where $g(t) = \tilde f_t(\ast)$.
Hence, the image of the corresponding fibered function $f: I \times X \to I \times
\R$ is just the graph of $g$.
Furthermore, since $\ast$ is a critical point of
$\tilde f_t$ for all $t$, the Cerf diagram of $f$ coincides with the graph of $g$.

For example, let $g(t) = 4t(1-t)$, plotted in Figure~\ref{fig:pmod_func}.
For convenience we will instead consider the piecewise linear function
$\tilde{g}(t) = 2 \min(t,1-t)$. This function is no longer smooth in $t$, but
its simplicity will make it easier to give a complete analysis
% (see Remark~\ref{rem:pl_approx}).
(see the comment following Eq.~\eqref{eqn:family_of_func}).

We have a diagram of topological spaces $F: \uD \times \uR \to \Top$ given by
$F(a,b,c) = f^{-1}([a,b] \times (-\infty,c])$, where $f:I \times X \to I \times
\R$ is given by $(t,*) \mapsto (t,\tilde{g}(t))$.
The space $F(a,b,c)$ is empty if $c < 2\min(a,1-b)$.
That is, $\frac{c}{2} < a \leq b < 1 - \frac{c}{2}$.
The space $F(a,b,c)$ is contractible if $c \geq 1$ or
if $2\min(a,1-b) \leq c < 2\max(a,1-b)$.
Equivalently, $a \leq \frac{c}{2}$ and $b < 1 - \frac{c}{2}$, or
$\frac{c}{2} < a$ and $1-\frac{c}{2} \leq b$.
In the remaining case, $2\max(a,1-b) \leq c < 1$, we find $F(a,b,c) \simeq S^0$, two disjoint points.
That is, $0 \leq c < 1$,  $0 \leq a \leq \frac{c}{2}$ and $1-\frac{c}{2} \leq b \leq 1$.

The persistence module $H_0F: \uD \times \uR \to \vect$ satisfies
\begin{equation}
  \label{eqn:pmod_func_cases}
  \dim H_0F =
  \begin{cases}
    1 & \text{if } c \geq 1 \\
    2 & \text{if } 2\max(a,1-b) \leq c < 1 \\
    1 & \text{if } 2\min(a,1-b) \leq c < 2\max(a,1-b) \\
    0 & \text{if } c < 2\min(a,1-b) \, ,
  \end{cases}
\end{equation}
while the persistence modules $H_j F$ are the trivial $K$-vector space for all $j > 0$.
%The support of this persistence module
See Figure~\ref{fig:pm-hat} for a visualization of $\beta_0$.

%%%%%%%%%%%%%%%%%%%%%%%%%%%%%%%%%%%%%%%%%%%%%%%%%
%% Graphs of functions, persistence modules.
%%%%%%%%%%%%%%%%%%%%%%%%%%%%%%%%%%%%%%%%%%%%%%%%%
\begin{figure}[h]
  \centering
%  \begin{subfigure}[b]{0.35\textwidth}
%    \centering
  \begin{minipage}{0.2\linewidth}
        \begin{tikzpicture}[scale=2]
      \draw[domain=0:1,smooth,thick] plot (\x,{-(2*\x-1)^2+1});
      \draw (0,0) -- (1,0);
      \draw[->] (0,0) -- (0,1.3);
    \end{tikzpicture}
  \end{minipage}
  \begin{minipage}{0.2\linewidth}
    \begin{tikzpicture}[scale=2]
      \draw[domain=0:0.5,smooth,thick] plot (\x,{2*\x});
      \draw[domain=0.5:1,smooth,thick] plot (\x,{-(2*\x)+2});
      \draw (0,0) -- (1,0);
      \draw[->] (0,0) -- (0,1.3);
    \end{tikzpicture}
    \end{minipage}
%  \caption{The graphs of $g(t) = 4t(1-t)$ and $\tilde{g}(t) = 2\min(t,1-t)$ %for $t \in [0,1]$.}
%    \label{fig:plot_ord_fun}
%  \end{subfigure}\quad
%  \begin{subfigure}[b]{0.55\textwidth}
%    \centering
    \begin{minipage}{0.35\linewidth}
    \begin{tikzcd}[cramped, column sep=-5pt]%, row sep = tiny]
       & F(0,1,1) &  \\
       & F(0,1,\tfrac{1}{2}) \ar[u] &  \\
      F(0,\tfrac{1}{2},\tfrac{1}{2}) \ar[ur] & & F(\tfrac{1}{2},1,\tfrac{1}{2}) \ar[ul] \\
                                             & F(\tfrac{1}{2},\tfrac{1}{2},0) \ar[ul] \ar[ur] &
    \end{tikzcd}
  \end{minipage}
  \begin{minipage}{0.2\linewidth}
    \begin{tikzcd}[column sep=tiny]%, row sep = tiny]
       & k &  \\
       & k^2 \arrow{u}{\tiny
         \begin{pampmatrix}
           1 & 1 
       \end{pampmatrix}} &  \\
       k \arrow{ur}{\tiny
         \begin{pampmatrix}
           1 \\ 0
         \end{pampmatrix}} & & k \arrow{ul}[swap]{\tiny
         \begin{pampmatrix}
           0 \\ 1
       \end{pampmatrix}} \\
       & 0 \arrow{ul} \ar[ur] &
    \end{tikzcd}
\end{minipage}
    % \caption{A subdiagram of $F$ corresponding to $\tilde{g}$ and the %corresponding persistence module.}
%    \label{fig:pers_ord_fun}
%  \end{subfigure}
  \caption{Left: the graph of $g(t) = 4t(1-t)$ for $t \in [0,1]$.
      Middle left: the graph of $\tilde{g}(t) = 2\min(t,1-t)$ for $t \in [0,1]$.
      This is also the image of the map $f:I \times \{*\} \to I \times \R$  given by $f(t,*) = (t,\tilde{g}(t))$.
    Middle right: a subdiagram of the diagram of spaces $F: \cat{\Delta^2 \times R} \to \cat{\Top}$ given by $F(a,b,c) = f^{-1}([a,b] \times (-\infty,c])$. 
    Right: the corresponding subdiagram of the persistence module $H_0 F: \cat{\Delta^2 \times R} \to \cat{Vect}$.} %
  \label{fig:pmod_func}
\end{figure}
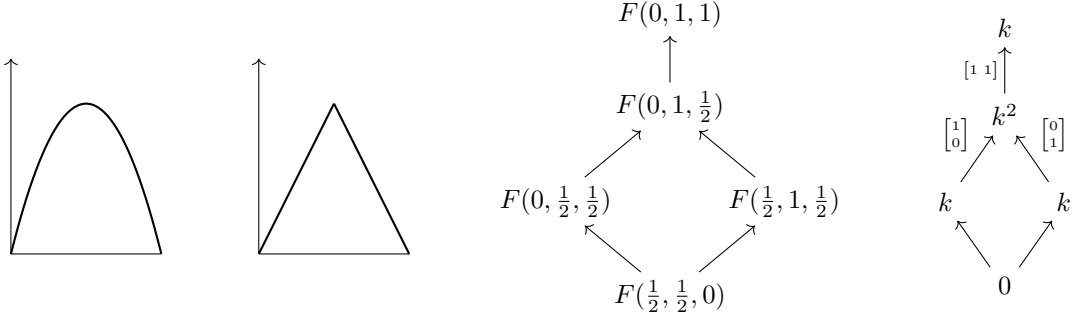

%%%%%%%%%%%%%%%%%%%%%%%%%%%%%%%%%%%%%%%%%%
%% Persistence module of ordinary function
%%%%%%%%%%%%%%%%%%%%%%%%%%%%%%%%%%%%%%%%%%
\begin{figure}
  \centering
  \tdplotsetmaincoords{70}{120}
  \begin{tikzpicture}[tdplot_main_coords]
%    \node at (1,-1,1) {$H_*F :$};
    % Draw and label axes
    \draw[->] (0,0,0) -- (-2,0,0) node[anchor=north west, left] {$a$};
    \draw[->] (0,0,0) -- (0,6,0) node[anchor=north west, below right]{$c$};
    \draw[->] (0,0,0) -- (0,0,3) node[anchor=south]{$b$};
    % Define how big "dent" from corner will be
    \newcommand\dent{.33}
    \draw[black] (.2,0,0) -- (-.2,0,0);
    \node[below left] at (0,0,0) {$0$};
    \draw(.2,3,0) -- (-.2,3,0);
    \node[below left] at (0,3,0) {$1$};
    \draw[color =dblue,dashed] (-2,2,2) -- (-2,4,2);
    \draw[color =dblue,dashed] (0,2,2)--(0,4,2);
    \draw[color =dblue,dashed] (0,2,0)--(0,4,0);
  % Draw angled lines to the corner of the 'final' square at c=3 (c=1 in rendered figure)
  \draw[color=dblue,opacity=.5] (0,0,0) -- (-1,3,3*\dent);
  \draw[color=dblue,opacity=.5] (-2,0,2) -- (-1,3,3*\dent);
  \draw[color=dblue, opacity=.5] (0,0,2) -- (-1,3,3*\dent);
% draw final square at c=3, (c=1 in rendered figure)
  \draw[color=dblue,densely dotted] (0,3,2) -- (0,3,3*\dent) -- (-1,3,3*\dent) -- (-1,3,2);
    % Fill the sides with color
    \fill [cblue,fill opacity = .3](0,2,0) -- (0,3,0) -- (0,3,2) -- (0,2,2);
    \fill [cblue,fill opacity = .3] (-2,2,2) -- (-2,3,2) -- (0,3,2) -- (0,2,2);
    \shade [left color=cblue, right color = white,opacity=.3] (0,3,0) -- (0,3,2) -- (0,4,2) -- (0,4,0);
  \shade [left color=cblue, right color = white,opacity=.3,shading angle=50] (0,3,2) -- (0,4,2) -- (-2,4,2) -- (-2,3,2); 
%%%%%%% Fix the shading angle here. See 
%%%% https://tex.stackexchange.com/questions/38418/how-to-use-shading-in-tikz-when-rotating-the-picture
% Draw the background triangles
\draw[color=gray!50] (0,0,0) -- (0,0,2) -- (-2,0,2) -- (0,0,0);
\draw[color=gray!50] (0,1,0) -- (0,1,2) -- (-2,1,2) -- (0,1,0);
\draw[color=gray!50] (0,2,0) -- (0,2,2) -- (-2,2,2) -- (0,2,0);
\draw[color=gray!50] (0,3,0) -- (0,3,2) -- (-2,3,2) -- (0,3,0);
% Draw shapes
\draw[fill=cblue,fill opacity=.5,draw=dblue] (0,1,2) -- (0,1,0) -- (-\dent,1,\dent) -- (-\dent,1,2) -- (0,1,2); %% at 1/3 front left trapezoid
\draw[fill=cblue, fill opacity=0.5, draw=dblue] (0,2,2) -- (-2*\dent,2,2) -- (-2*\dent,2,2*\dent) -- (0,2,0) -- (0,2,2); %% at 2/3, front left trapezoid
\draw[fill=cblue,fill opacity=0.5,draw=dblue] (0,2,2) -- (0,2,2-2*\dent)-- (-2+2*\dent,2,2-2*\dent) -- (-2,2,2) -- (0,2,2); % at 2/3 top right trapezoid
\draw[fill = cblue,fill opacity=0.5,draw=dblue] (0,1,2) -- (0,1,2-\dent) -- (-2+\dent,1,2-\dent) -- (-2,1,2) -- (0,1,2); % at 1/3, top right trapezoid
\draw[fill=cblue, fill opacity = .5,draw=dblue] (0,0,2) -- (-2,0,2) -- (-2,3,2) -- (0,3,2) -- (0,0,2); %% top rectangle
\draw[fill=cblue,fill opacity=.5,draw=dblue] (0,0,2) -- (0,0,0) -- (0,2,0) -- (0,2,2) -- (0,0,2); %front rectangle
\draw[fill=cblue,fill opacity =.5,draw=dblue] (0,0,2) -- (0,3,3*\dent) -- (0,3,2) --(0,0,2); %% front triangle 0 -> 2/3
\draw[fill=cblue, fill opacity=.5, draw = dblue] (0,0,2) -- (-3*\dent,3,2) -- (0,3,2) -- (0,0,2); %% top triangle 0 -> 2/3
% Draw lines in the front
\draw[color =dblue] (0,0,2)--(0,2,2);
\draw[color=dblue](0,0,2) -- (-2*\dent,2,2);
\draw[color=dblue] (0,0,0) -- (0,0,2);
\draw[fill=cblue,fill opacity = .5,draw=dblue] (0,3,2) -- (0,3,0) -- (-2,3,2) -- (0,3,2);
\node[scale=0.5] at (-0.3,2,0.7) {$k$};
\node[scale=0.5] at (-1.5,2,1.8) {$k$};
\node[scale=0.5] at (-0.3,2,1.7) {$k^2$};
\node[scale=0.5] at (-0.3,3,1.7) {$k$};
\draw[->] (-0.3,2,0.9) -- (-0.3,2,1.5);
\draw[->] (-1.3,2,1.75) -- (-0.6,2,1.7);
\draw[->] (-0.3,2.3,1.7) -- (-0.3,2.7,1.7);
  \end{tikzpicture}
  \caption{The multiparameter persistence module $H_* F: \cat{\Delta^2 \times R} \to \cat{Vect}$ defined by $H_*F(a,b,c) = H_*(f^{-1}([a,b] \times (-\infty,c])$ where $f:I \times \{*\} \to I \times \R$ is given by $f(t,*) = (t,\tilde{g}(t))$ from \cref{fig:pmod_func}.
      For $j \leq 0$, $H_j(F) = 0$.
      We have $\beta_0 = 2$ in the open square pyramid
      given by $0 \leq c < 1$, $0 \leq a \leq \frac{c}{2}$ and $1 - \frac{c}{2} \leq b \leq 1$.
      Furthermore, $\beta_0 = 1$ in the semi-infinite triangular cylinder given by $0 \leq a \leq b \leq 1$ and $c \geq 1$.
      For $0 \leq c < 1$, we also have $\beta_0 = 1$ in the region given by $0 \leq a \leq \frac{c}{2}$ and $a \leq b < 1-\frac{c}{2}$ and the region $\frac{c}{2} < a \leq b$ and $1-\frac{c}{2} \leq b \leq 1$.
      Everywhere else, $\beta_0 = 0$.
      That is, for $c < 0$ and $0 \leq a \leq b \leq 1$ and for $0 \leq c < 1$ and $\frac{c}{2} < a \leq b < 1-\frac{c}{2}$.
      The right hand diagram in Figure~\ref{fig:pmod_func} 
      is embedded in $H_0F$ as indicated. It follows that $H_0F$ is indecomposable.}
  \label{fig:pm-hat}
\end{figure}

% Note that the persistence module $H_0F$ is indecomposable and has maximum dimension two.
The diagram of spaces $F$ has the sub-diagram given in Figure~\ref{fig:pmod_func},
which has a corresponding indecomposable persistence module, also in
\cref{fig:pmod_func}.
This submodule of $H_0 F$ is also visualized in \cref{fig:pm-hat}.
It follows that $H_0 F$ is an indecomposable persistence module.
%The blue regions correspond to the support of $H_0F$.

\subsection{Indecomposable persistence modules with arbitrary maximum dimension}
\label{sec:arb-dim}

The example in the previous section can be generalized to produce an indecomposable persistence module
arising from a one-parameter family of functions, with arbitrarily large maximum dimension.

For $n>0$, let $\tilde{g}_n:I \to \R$ be the piecewise linear function obtained by linear interpolation between the values
$\tilde{g}_n(\frac{i}{n}) = 0$ for $0 \leq i \leq n$ and
$\tilde{g}_n(\frac{2i-1}{2n}) = 1$ for $1 \leq i \leq n$ (the example of Section~\ref{sec:pmod_func} is the case $n=1$).
Then we have the corresponding diagram of topological spaces
$F: \uD \times \uR \to \Top$ given by $F(a,b,c) = f^{-1}([a,b] \times
(-\infty,c])$, where $f:I \times \{*\} \to I \times \R$ is given by $(t,*) \mapsto
(t,\tilde{g}(t))$.
Now $F$ has a finite subdiagram $\hat{F}$ given by $F(\frac{i}{n},\frac{j}{n},\frac{1}{2})$, where $0 \leq i \leq j \leq n$.

Applying $H_0$ we obtain the persistence module $H_0F$, which contains the following % discrete
persistence module $H_0\hat{F}$:
\begin{equation*}
  \begin{tikzcd}[row sep = tiny,column sep = tiny]
      &&&& k^{n+1} &&&&& \\
      &&& \iddots \ar[ur] & \cdots & \ddots \ar[ul] &&&& \\
      && k^3 \ar[ur] && k^3 \ar[ur]\ar[ul] & \cdots & k^3 \ar[ul] &&& \\
      & k^2 \ar[ur] && k^2 \ar[ur] \ar[ul] && k^2 \ar[ul] & \cdots & k^2 \ar[ul] && \\
    k \ar [ur] && k\ar[ur] \ar[ul] &&k \ar[ur] \ar[ul] && \cdots && k \ar[ul] &
  \end{tikzcd}
\end{equation*}
Each linear map $k^m \to k^{m+1}$ pointing up and to the right is given by the inclusion
$k^m \to k^m \oplus k \isom k^{m+1}$.
and each linear map $k^m \to k^{m+1}$ pointing up and to the left left is given by the inclusion
$k^m \to k \oplus k^{m} \isom k^{m+1}$.
This persistence module is decomposable into $(n+1)$ one-dimensional summands, whose support is given by the up-set of one of the $(n+1)$ minimal elements.

Now append the terminal element $F(0,1,1)$ to the diagram $\hat{F}$ to obtain the diagram $\check{F}$, which is also a subdiagram of $F$.
Then we have the persistence module $H_0\check{F}$,
\begin{equation*}
  \begin{tikzcd}[row sep = tiny,column sep = tiny]
      &&&& k &&&&& \\
      &&&& k^{n+1} \ar[u] &&&&& \\
      &&& \iddots \ar[ur] & \cdots & \ddots \ar[ul] &&&& \\
      && k^3 \ar[ur] && k^3 \ar[ur]\ar[ul] & \cdots & k^3 \ar[ul] &&& \\
      & k^2 \ar[ur] && k^2 \ar[ur] \ar[ul] && k^2 \ar[ul] & \cdots & k^2 \ar[ul] && \\
    k \ar [ur] && k\ar[ur] \ar[ul] &&k \ar[ur] \ar[ul] && \cdots && k \ar[ul] &
  \end{tikzcd}
\end{equation*}
where the linear map $k^{n+1} \to k$ is given by summing the coordinates. This persistence module is indecomposable
since the upset of every minimal element contains the terminal element $H_0F(0,1,1) \isom k$.

\subsection{A class of indecomposable persistence modules}
\label{sec:indecomposable}

Let $g:I \to \R$ be any (not necessarily continuous) bounded real-valued function on the unit interval. Let $f: I \times \{\ast\} \to I \times \R$ be given by $f(t,*) = (t,g(t))$ and let $F: \uD \times \uR \to \Top$ be given by $F(a,b,c) = f^{-1}([a,b] \times (-\infty,c])$.

\begin{theorem}
	\label{thm:one_param_indec}
  Let $f_t(*) = g(t)$ be any uniformly bounded one-parameter family of functions on a one point space $\{\ast\}$. Then the corresponding persistence module $H_{j}F$ is indecomposable for every $j \geq 0$.
\end{theorem}

\begin{proof}
  For all $(a,b,c)) \in \Delta^2 \times \R$, $F(a,b,c)$ deformation retracts to a subset of $I$, so
  $H_k(F) =0$ for $k \geq 1$. 
  Recall (\cref{sec:index-cat}) that for $p \in \Delta^2 \times \R$, $U_p$ denotes the up-set of $p$.
  
    Assume that $H_0F \isom M \oplus N$ is a nontrivial decomposition of $H_0F$.
    Then there are nonzero maps $p:H_0 F \to M$, $q:H_0 F \to N$, $i:M \to H_0 F$, and $j:N \to H_0 F$ such that $ip + jq = 1_{H_0 F}$.
  Choose $B$ $\in \R$ such that $g(t) \leq B$ for all $t \in I$.
  Let $T = (0,1,B)$.
  Then $(H_0 F)_T = k$. It follows that either $i_T$ or $j_T$ is the zero map. Assume without loss of generality that $i_T = 0$.

  % Let $p_0:H_0F \to M$ and $p_1:H_0F \to N$ denote the canonical projection maps and let $i_0:M \to H_0F$ and $i_1:N \to H_0F$ denote the canonical inclusion maps.
  % Choose $B \in \R$ such that $g(t) \leq B$ for all $t \in I$.

  By definition, we have that for all $t \in I$,
  \begin{equation*}
    F(t,t,c) =
    \begin{cases}
      (t,c) & \text{if } c \geq g(t) \\
      \emptyset & \text{if } c < g(t).
    \end{cases}
  \end{equation*}
  So, in particular $H_0F(t,t,g(t)) = k$.
  Furthermore, we have a surjection of persistence modules
  \begin{equation*}
    \bigoplus_{t \in I} k[U_{(t,t,g(t))}] \xto{\varphi} H_0F.
  \end{equation*}
  Since $\varphi$ is surjective and $p$ is nonzero, it follows that $p \circ \varphi$ is nonzero.
    Therefore there exists an $a = (t_0,t_0,g(t_0))$ such that $k[U_{a}] \xto{p\varphi} M$ is nonzero, 
    which forces $(p\varphi)_{a}: k[U_{a}]_{a} \to M_{a}$ to also be nonzero.
  % Thus $(p\varphi)_{a}: k[U_{a}]_{a} \to M_{a}$ is nonzero, since
  % $p(\varphi)$ is determined by $(p\varphi)_{a}$.
  Since $k[U_a]_a \isom k$ and $(H_0 F)_a \isom k$, it follows that $p_a: (H_0F)_a \to M_a$ is injective.
    Therefore, $q_a = 0$.

  Since $ip + jq = 1_{H_0 F}$, $(H_0 F)_{a \leq T} = i_T M_{a \leq T} p_a + j_T N_{a \leq T} q_a  = 0$, which is a contradiction.
 % 
  % Similarly, there exists a $t_1$ such that
  % $(p_1\varphi)_{(t_1,t_1,g(t_1))}: k[U_{(t_1,t_1,g(t_1))}]_{(t_1,t_1,g(t_1))} \to N_{(t_1,t_1,g(t_1))}$ is nonzero.
%
  % Next, we have the following commutative diagram, where $a = (t_0,t_0,g(t_0))$.
  % \begin{equation*}
  %   \begin{tikzcd}[column sep=huge]
  %     k[U_{a}]_{a} \ar[r,"\varphi_{a}"] &
  %     H_0 F_{a} \ar[r,"H_0 F_{a \leq (0,1,B)}"] \ar[d,"(p_0)_{a}"'] &
  %     H_0 F_{(0,1,B)} \\
  %     & M_{a} \ar[r,"M_{a \leq (0,1,B)}"] & M_{(0,1,B)} \ar[u,"(i_0)_{(0,1,B)}"'] 
  %   \end{tikzcd}
  % \end{equation*}
  % Since $(p_0)_a\varphi_a$ is nonzero and $k[U_a]_a \isom k$, it follows that $\varphi_a$ is injective. Furthermore the top right horizontal map is an isomorphism on $k$ induced by the inclusion of $t_0$ in $[0,1]$.
  % Therefore, the map $(i_0)_{(0,1,B)}$ is a surjection.
  % This map is induced by the inclusion $\{t_0\} \incl I$.
  % So it is an isomorphism.
  % Similarly, the composition
  % \begin{equation*}
  %   k \isom k[U_{(t_1,t_1,g(t_1))}]_{(t_1,t_1,g(t_1))} \xto{(p_1\varphi)_{(t_1,t_1,g(t_1))}} N_{(t_1,t_1,g(t_1))} \xto{N_{(t_1,t_1,g(t_1)) \leq (0,1,B)}} N_{(0,1,B)} \xto{(i_1)_{(0,1,B)}} H_0F_{(0,1,B)} \isom k
  % \end{equation*}
  % is an isomorphism.
  % Together, these two isomorphisms contradict $H_0F \isom M \oplus N$.
\end{proof}

\subsection{The cylinder}
\label{sec:cyl}

Increasing the dimension of the manifold in our examples, consider $X = S^1$.
Let $\tilde{f}_t: S^1 \to \R$ be the constant family of height functions on the circle;
$\tilde{f}_t(\theta) = \sin \theta$.
The corresponding function fibered over the interval
$f:I \times S^1 \to I \times \R$
has domain the cylinder, % $I \times S^1$;
and is given by $f(t,\theta) = (t,\sin \theta)$.
See Figure~\ref{fig:std_cyl}.
%  The function $\tilde f$ is vertical projection,
%  i.e., constant with respect to $t \in I$.
The corresponding Cerf diagram is shown in Figure~\ref{fig:std_cyl_cerf}. It consists of two horizontal lines, corresponding to the fold singularities
 of $\tilde{f}$
given by the global minimum and global maximum of the height function.

% We give an explicit description of the space $F([a,b];c)$ for $[a,b] \subset
% [0,1]$, and $c \in \R$.

% \begin{enumerate}
%   \item If $c<u$, the space $F([a,b];c)$ is empty. 
%   \item If $u \leq c < v$, the space $F([a,b];c)$ equals $[a,b] \times A$, where $A$ is an arc of the circle, and is hence contractible. 
%   \item If $c \geq v$, then $F([a,b];c)$ equals the cylinder $[a,b] \times S^1$.
% \end{enumerate}

By definition, $F(a,b,c) = f^{-1}([a,b] \times (-\infty,c]) = [a,b] \times \{ \theta \ | \ \sin \theta \leq c\}$.
Therefore $F(a,b,c)$ is empty if $c < 0$, $F(a,b,c)$ is contractible if $0 \leq c < 1$, and $F(a,b,c)$ is homotopy equivalent to $S^1$ if $c \geq 1$.
Thus we find
\begin{equation*}
  \dim H_0F =
  \begin{cases}
    1 &\text{if } c \geq 0 \\
    0 &\text{if } c < 0 \, ,
  \end{cases}
  \quad \text{and } \quad
  \dim H_1F =
  \begin{cases}
    1 &\text{if } c \geq 1 \\
    0 &\text{if } c < 1.
  \end{cases}
\end{equation*}

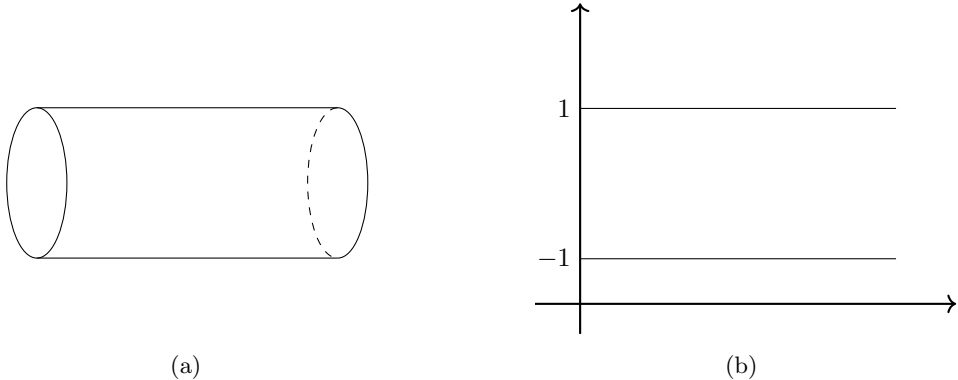
\begin{figure}[h]
  \centering
  %    \scalebox{0.9}{
  \begin{subfigure}[t]{0.45\textwidth}
    \centering
    \raisebox{0.4in}{
      \begin{tikzpicture}[scale=0.8] % cylinder
        \draw (0,0) ellipse (0.5 and 1.25); % left cap
        \draw (0,-1.25) -- (5, -1.25); % bottom horizontal line
        \draw (5,-1.25) arc (-90:90:0.5 and 1.25); % right cap, right
        \draw [dashed] (5,1.25) arc (90:270:0.5 and 1.25); % right cap, left
        \draw (5,1.25) -- (0,1.25); % top horizontal line
      \end{tikzpicture}
    } %% end raisebox
    %	\caption{The cylinder}
    \caption{}
    \label{fig:std_cyl}
  \end{subfigure}%
  % } %% end scalebox
  ~
  %  \begin{scalebox}{0.9}{
  \begin{subfigure}[t]{0.45\textwidth}
    \centering
    \begin{tikzpicture}[scale=0.8]
      \draw[->,thick] (-1.25,-2.5) -- (-1.25,3);
      \draw[->,thick] (-2,-2) -- (5,-2);
      \draw (-1.25,-1.25) -- (4, -1.25); % bottom horizontal line
      \draw (-1.25,1.25) -- (4,1.25); % top horizontal line
      \node [left] at (-1.25,-1.25) {$-1$};
      \node [left] at (-1.25, 1.25) {$1$};
    \end{tikzpicture}
    %	\caption{The Cerf diagram of the cylinder.}
    \caption{}
    \label{fig:std_cyl_cerf}
  \end{subfigure}
  % } 
  \caption{Left: The cylinder $I \times S^1$. Right: the Cerf diagram of the constant one-parameter family of height functions on the circle.}
  \label{fig:cylinder}
\end{figure}

This multiparameter persistence module can also be visualized as shown in
Figure~\ref{fig:pm-cylinder}. The blue region is the support of $H_0F$ and the red
region is the support of $H_1F$. These two regions are unbounded, analogous to sub-level set
persistence of the standard height function on $S^1$. Since each of $H_0F$ and $H_1F$ are
indecomposable and at most one-dimensional, this visualization also shows
the structure of the multiparameter persistence module $H_*F$.

\begin{figure}[h]
  \centering
  \tdplotsetmaincoords{70}{120}
  \begin{tikzpicture}[tdplot_main_coords,scale=0.8]
%    \node at (1,-1,1) {$H_*F :$};
    % Draw and label axes
    \draw[->] (0,0,0) -- (-2,0,0) node[anchor=north east,left] {$a$};
    \draw[->] (0,0,0) -- (0,8,0) node[anchor=north west, below right]{$c$};
    \draw[->] (0,0,0) -- (0,0,3) node[anchor=south]{$b$};
    % Define how big "dent" from corner will be
    \newcommand\dent{.2}
    %%%%%%%%%%%%%%%%
    %%%%   Draw the infinite bars
    %%%%%%%%%%%%%%%%%%%%
    \begin{scope}[xshift=5.5ex,yshift=-1.2ex]
      % \draw[gray!50, dashed] (2,0,0) -- (-2,0,0);
      \draw[black] (.2,0,0) -- (-.2,0,0);
      \node[below left] at (0,0,0) {$-1$};
      \draw(.2,4,0) -- (-.2,4,0);
      \node[below left] at (0,4,0){$1$};
      % \begin{scope}[xshift=-6ex,yshift=-15ex,tdplot_main_coords]
      %%%%%%%%%%%%% H_0 bar
      % Draw the filled-in triangles
      \draw[fill=cblue,fill opacity=.8,draw=dblue] (0,0,0) -- (0,0,2) -- (-2,0,2) -- (0,0,0);
      \draw[fill=cblue,fill opacity=.8,draw=dblue] (0,1,0) -- (0,1,2) -- (-2,1,2) -- (0,1,0);
      \draw[fill=cblue,fill opacity=.8,draw=dblue] (0,2,0) -- (0,2,2) -- (-2,2,2) -- (0,2,0);
      % \draw[fill=cblue,fill opacity=.8,draw=dblue] (0,3,0) -- (0,3,2) -- (-2,3,2) -- (0,3,0);
      % Draw the lines in front
      \draw[color =dblue] (-2,0,2)--(-2,2,2);
      \draw[color =dblue] (0,0,2)--(0,2,2);
      \draw[color =dblue] (0,0,0)--(0,2,0);
      \draw[color =dblue,dashed] (-2,2,2) -- (-2,3.5,2);
      \draw[color =dblue,dashed] (0,2,2)--(0,3.5,2);
      \draw[color =dblue,dashed] (0,2,0)--(0,3.5,0);
      % Fill the sides with color
      \fill [cblue,fill opacity = .3](0,0,0) -- (0,2,0) -- (0,2,2) -- (0,0,2);
      \fill [cblue,fill opacity = .3] (-2,2,2) -- (-2,0,2) -- (0,0,2) -- (0,2,2);
      \shade [left color=cblue, right color = white,opacity=.3] (0,2,0) -- (0,2,2) -- (0,3.5,2) -- (0,3.5,0);
    \shade [left color=cblue, right color = white,opacity=.3,shading angle=50] (0,2,2) -- (0,3.5,2) -- (-2,3.5,2) -- (-2,2,2); 
  %%%%%%% Fix the shading angle here. See 
  %%%% https://tex.stackexchange.com/questions/38418/how-to-use-shading-in-tikz-when-rotating-the-picture
  %%%%%%%%% H_1 bar
  % tick marks
  % Draw the filled-in triangles
  \draw[fill=Red,fill opacity=.2,draw=Red] (0,4,0) -- (0,4,2) -- (-2,4,2) -- (0,4,0);
  \draw[fill=Red,fill opacity=.2,draw=Red] (0,5,0) -- (0,5,2) -- (-2,5,2) -- (0,5,0);
  \draw[fill=Red,fill opacity=.2,draw=Red] (0,6,0) -- (0,6,2) -- (-2,6,2) -- (0,6,0);
  % Draw the lines in front
  \draw[color =Red] (-2,4,2)--(-2,6,2);
  \draw[color =Red] (0,4,2)--(0,6,2);
  \draw[color =Red] (0,4,0)--(0,6,0);
  \draw[color =Red,dashed] (-2,6,2) -- (-2,7,2);
  \draw[color =Red,dashed] (0,6,2)--(0,7,2);
  \draw[color =Red,dashed] (0,6,0)--(0,7,0);
  % Fill the sides with color
  \fill [Red,fill opacity = .4](0,4,0) -- (0,6,0) -- (0,6,2) -- (0,4,2);
  \fill [Red,fill opacity = .4] (-2,6,2) -- (-2,4,2) -- (0,4,2) -- (0,6,2);
  \shade [left color=Red, right color = white,opacity=.4] (0,6,0) -- (0,6,2) -- (0,7,2) -- (0,7,0);
\shade [left color=Red, right color = white,opacity=.4,shading angle=70] (0,6,2) -- (-2,6,2) -- (-2,7,2) -- (0,7,2); 
      % \end{scope}
    \end{scope}
  \end{tikzpicture}
  \caption{The multiparameter persistence module of the constant $1$-parameter family of height functions on the circle.
      The Betti functions are constant for $0 \leq a \leq b \leq 1$. $\beta_0=0$ for $c < -1$ and $\beta_0=1$ for $c \geq -1$ (blue).  $\beta_1 = 0$ for $c < 1$ and $\beta_1=1$ for $c \geq 1$ (red).
  For $H_0 F$, all linear maps within the blue region are the identity map. Similarly, for $H_1 F$, all linear maps with the red region are the identity map. That is, both $H_0 F$ and $H_1 F$ are one-dimensional persistence modules supported on semi-infinite triangular prisms in which all non-trivial maps are identity maps.}
  \label{fig:pm-cylinder}
\end{figure}
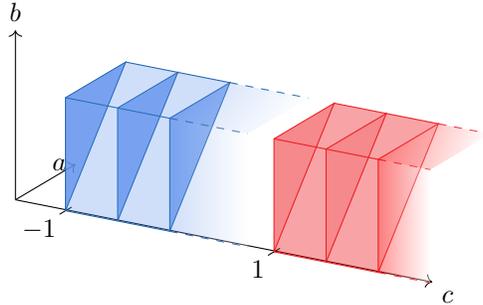

% \subsection{\notyet Two wrinkles}
% \label{sec:two-wrinkles}

% We now modify the example of Section~\ref{sec:} by introducing an additional wrinkle into the
% cylinder. 

% \mc{Make remark about how these examples don't decompose in 
% such a nice way that one might hope. This hopefully will be 
% exemplified by the two wrinkle example, or whatever ends up
% going in this section.}

%%% Local Variables:
%%% mode: latex
%%% TeX-master: "1-parameter"
%%% End:

% \section{\notyet Theory}
% \input{theory}

\section{Analyzing diagrams of spaces}
\label{sec:extra}
Let $X$ be a smooth, compact manifold.
%Fold singularities are not isolated
For generic 1-parameter
families of
smooth
functions $\tilde{f}: I \times X \to \R$,
the nondegenerate critical points of the fibers $\tilde{f}_t$ occur
%Rather, they arise
in families themselves, as can be seen in the arcs of the Cerf
diagrams of Section~\ref{sec:examples}.
% The generic nature of fold singularities
% for a 1-parameter family of functions $\tilde f: I \times X \to \R$ implies
% fold singularities come in families. For example, the arcs in any of the Cerf
% diagrams of Section~\ref{sec:examples} are the images of these families under
% $\tilde f$.

\begin{definition} \label{def:positive-negative}
We say that %a fold singularity
such a critical point
is \emph{positive} if % its image
the curve
in the Cerf
diagram
containing its value
has positive slope (is locally strictly increasing).
% Say that a fold singularity is
Similarly, say that such a critical point is
\emph{negative} if
% its image in the Cerf diagram
the curve in the Cerf diagram containing its  value
has negative slope (is locally
decreasing). There can be points %at which the family is
that are neither positive nor
negative, e.g., the maximum or minimum of the singular locus of a wrinkle. 
\end{definition}

Recall that $f:I \times X \to I \times \R$ is given by $f(t,x) =
(t,\tilde{f}(t,x)) = (t,\tilde{f}_t(x))$ and $F: \cat{\Delta^2 \times R} \to
\cat{Top}$ is given by $F(a,b,c) = f^{-1}([a,b] \times (-\infty,c])$.
% Let $\cat{R_{=}}$ denote the category with corresponding to $\R$ together with the partial order given by equality. It has objects $\R$ and only identity morphisms.
% Define $F_0: \cat{\Delta^2 \times R_{=}} \to \cat{Top}$ given by
For $(a,b,c) \in \Delta^2 \times \R$, let
$F_0(a,b,c) = f^{-1}([a,b] \times \{c\})$.

\begin{theorem}
  \label{prop:morse_proj}
  Suppose $\tilde f:I \times X \to \R$ is a generic 1-parameter family of functions on a 
  smooth, compact manifold $X$.
  Let $0 \leq a < b \leq 1$ and $c \in \R$.
  Then $(F(a,b,c),F_0(a,b,c))$ is a cobordism between $(F(a,a,c),F_0(a,a,c))$ and $(F(b,b,c),F_0(b,b,c))$.
  % Assume that the `corners'
  %    of $[a,b] \times (-\infty,c]$ ($(a,c)$ and $(b,c)$)
  %    do not intersect $\Gamma_{\tilde f}$.
  Assume that
    % $c$ is not a critical value of $\tilde{f}_a$ or $\tilde{f}_b$.
  there are no critical points in $F_0(a,a,c)$ and $F_0(b,b,c)$.
  % \mc{We do not want critical points for $\pi_I$ on the boundary manifolds and this only
  %   happens when the top strip meets the side strips, i.e., the corner points.} 
  % If you think
  %   this is a clunky description, here are alternatives:
  %   i) Write $\{a\} \times \{c\}$ and $\{b\} \times \{c\}$ instead of ordered pairs,
  % ii) Assume that $\tilde f_a(\Sigma(\tilde f_a)) < c$ and 
  % $\tilde f_b(\Sigma(\tilde f_b) < c$, or 
  % iii) $\Gamma_{\tilde f} \cap (\{a,b\} \times \{c\}) = \emptyset$.
  % Assume that $\tilde{f}_a$ and $\tilde{f}_b$ each have finitely many fold singularities with distinct critical values and no other singularities.
  Then the projection onto $[a,b]$
  \begin{equation}
    \pi_{[a,b]} : F(a,b,c) \to [a,b]
  \end{equation}
  is a Morse function on the cobordism $(F(a,b,c),F_0(a,b,c))$.
  Furthermore, $\pi_{[a,b]}$ has no interior critical points.
  In addition, %intersections of
  positive and negative %fold singularities in $\Sigma^{10}(\tilde f)$ with
  critical points in
  $F_0(a,b,c)$
  are boundary stable and boundary unstable
  critical points of $\pi_{[a,b]}$, respectively.
\end{theorem}

\begin{proof}
  The projection $\pi_{[a,b]}: [a,b] \times X \to [a,b]$ is a 
  submersion and hence has no
  critical points. Therefore, all critical points of the restriction
  $\pi_{[a,b]}: F(a,b,c) \to [a,b] \subset I$ must lie on the boundary
  % component $Y = \tilde f^{-1}(c)$.
  $Y= F_0(a,b,c)$.

  Consider a %fold singularity $
  nondegenerate critical point $z$ of $\tilde{f}_t$
%  z \in \Sigma^{10}(\tilde{f})$
  with $z \in Y$.
  Near this %fold singularity,
  point
  there exists a 
  coordinate system
  for which it is a fold singularity of $\tilde{f}$ given by
%    of the form prescribed by
    Eq.~\eqref{eqn:fold_sing}.  In this coordinate system,
  $\pi_{[a,b]}(t,x) =t$, so $d\pi_{[a,b]} = [\, 1 \, 0 \, \cdots \, 0 \, ]$.
  % Assume $z \in \Sigma^{10}(\tilde{f}) \cap F_0(a,b,c)$ and for simplicity, further
  For simplicity, assume that
%  $\tilde f(z) = 0$ by translating to the origin.
  $(z,f(z))$ is at the origin.
  Then the level set $\{0\} \times \{x_1^2 + \cdots + x_j^2 - x_{j+1}^2 - \cdots - x_{n-1}^2 = 0\}$
  has tangent space contained in $\{0\} \times \R^{n-1}$ and hence lies $\ker d\pi_{[a,b]}$. Therefore,
  % the intersection of $\Sigma^{10}(\tilde f)$ and $F_0(a,b,c)$ consists of critical points
  $z$ is a critical point
  of $\pi_{[a,b]}$.

  % In this coordinate system, $\pi_I(x_1,x_2,\ldots,x_n,t) = t$, so that
  % $\nabla \pi_I = [ \, 0 \, 0 \, \cdots \, 0 \, 1 \, ]$ for a suitable choice of
  % Riemannian metric. 

  Suppose %$z \in \Sigma^{10}(\tilde f)$ is a negative fold singularity
  $z$ is a negative critical point of $\tilde{f}_t$
  and $z
  \in Y$. There exists a path $\alpha: \R \to %\Sigma^{10}(\tilde f) \subset
  I \times X$ 
whose image consists of points $(t,x)$ where $x$ is a critical point of $\tilde{f}_t$
 so that $\alpha(0) = z$, $\tilde f(\alpha(-t)) > c$ and $\tilde
  f(\alpha(t)) < c$ for $t > 0$ (i.e., a parametrization of the preimage of an arc in
  the Cerf diagram containing the image of $z$). Thus, $\alpha$ restricted to
  $[0,\infty)$ provides a path in $F(a,b,c)$ so that $f(\alpha(t)) \not \in Y$
  for $t>0$, and therefore, $T_z W^u_z \not \subset T_z Y$. Hence, $z$ is a
  boundary unstable critical point for $\pi_{[a,b]}$.

  On the other hand, suppose
  % $z \in \Sigma^{10}(\tilde f)$ is a positive fold singularity 
  $z$ is a positive critical point of $\tilde{f}_t$ and $z \in Y$.
  Near $z$,
    there exists a coordinate system on $I \times X$ of the form prescribed by
    Eq.~\eqref{eqn:fold_sing}. In this coordinate system, we find $\pi_{[a,b]}(t,x) = t$
    and thus $d \pi_{[a,b]} = [\, 1 \, 0 \, \cdots \, 0 \, ]$.
    Since $z$ is a positive %fold singularity,
    critical point,
    If we take a sufficiently small such neighborhood $U$, then
    the $\tilde f$ function values will increase along the flow lines of 
    $\nabla \pi_{[a,b]}$. Precisely, if $\xi: \R \times I \times X \to I \times X$
    denotes the flow generated by $\nabla \pi_{[a,b]}$, then we have 
    $\tilde f( \xi (\epsilon,t,x) ) \geq \tilde f (t,x)$ for $\epsilon \geq 0$ and $(t,x) \in U$.
    This inequality holds in the restriction to $U \cap F(a,b,c)$.
    Hence, points on $Y \cap U = F_0(a,b,c) \cap U$ must flow to other
    points on $Y$ under $\xi$. In particular, $U \cap W^s_z \subset Y$ and
    hence $T_z W^s_z \subset T_z Y$.
\end{proof}

  \begin{remark}
    \cref{prop:morse_proj} does not address the case when
    % the fold singularity intersects
    $F_0(a,b,c)$ contains nondegenerate critical points that are neither positive nor negative
    % with zero slope
    or the case that $F_0(a,b,c)$ contains cusp singularities.
  \end{remark}

% If we post-compose $\pi_I$ with the involution $t \mapsto 1-t$, then the
% flow of this composition is generated by $-\nabla \pi_I$ and the stable and
% unstable manifolds swap (see Remark~\ref{rem:flow_swap}). Correspondingly,
% the positive or negative label on each fold singularity also swaps. Thus
% the positive fold singularity intersection case follows from the negative
% fold singularity case.  \pb{Why? You've shown that $T_z W_z^u \not\subset
% T_z Y$. I don't see that it follows that $T_z W_z^s \subset T_z Y$.}

The remainder of this section and the next section are dedicated
  to showing how the theory developed thus far and in particular, Theorem~\ref{prop:morse_proj},
  can be applied to examples.

  \begin{example}
    Consider the function $\tilde{g}(t) = 2\min(t,1-t)$ in 
%    The function $\tilde g:I \times \{\ast\} \to \R$ of
    Section~\ref{sec:pmod_func}.
    Let $X = \{*\}$
    and define $\tilde{f}:I \times X \to \R$ to be given by $\tilde{f}(t,*) = \tilde{g}(t)$
    and define $f:I \times X \to I \times \R$ to be given by $f(t,*) = (t,\tilde{g}(t))$.
    For $f$, we have the associated
    diagram of topological spaces $F:
    \uD \times \uR \to \Top$.

    By definition, $\ast$ is a critical point of
    $\tilde g_t$ for all $t \in I$.
    %and $(t,\ast)$ is a fold singularity of $f$.
    % : I \times \{\ast\} \to I \times \R$, where $f(t,\ast) = (t, \tilde     g_t(\ast))$.
    Let $0 \leq a < b \leq 1$ and let $c \in \R$.
    In the case
    % when $a < 1-b$ and $c$ satisfies
    % $2a \leq c \leq 2-2b$
    % (see Eq.~\eqref{eqn:pmod_func_cases}), we find
    $a < \frac{1}{2}$ and $2a < c < \tilde{g}(b)$, then
    $F_0(a,b,c)$ %intersects $\Sigma(\tilde{f})$
    % in a positive fold singularity
    has a positive critical point
    and 
    Proposition~\ref{prop:morse_proj} implies this intersection coincides with
    a boundary stable critical point of $\pi_{[a,b]}$. Proposition~\ref{prop:morse-thy-bdy-cob}
    implies $F(a,b,c)$ is a left-product cobordism.
%    , as it must be in this simple case since
    Note that $F(b,b,c)$ is empty.
    % A similar analysis of the case when 
    % $a > 1-b$ and $c$ satisfies $2(1-b) \leq c < 2a$ yields
    {Similarly, if $b > \frac{1}{2}$ and $2(1-b) < c < \tilde{g}(a)$
      then we have}
    a single boundary
    unstable critical point for $\pi_{[a,b]}$ and $F(a,b,c)$ is a right-product cobordism.
%
%    Consider the case when $a > 1-b$ and $c$ satisfies $2a \leq c < 1$, so
%    that
    In the case that $a < \frac{1}{2} < b$ and $2a, 2(1-b) < c < 1$, we have that
    $F(a,b,c)$ is a cobordism between the singletons $F(a,a,c)$ and
    $F(b,b,c)$. This is not a product cobordism, however, since the projection
    $\pi_{[a,b]}$ has both a boundary stable and a boundary unstable critical point.
%   In fact, this is obvious from visualizing $F(a,b,c)$, since $F(d,e,c)$ is
%   empty for any $d$ and $e$ satisfying $2a < d <e < 1-b$.
    Note that $F(\frac{1}{2},\frac{1}{2},c)$ is empty.
\end{example}

\section{Examples II: the wrinkled cylinder}
\label{sec:wrinkled-cylinder}
\input{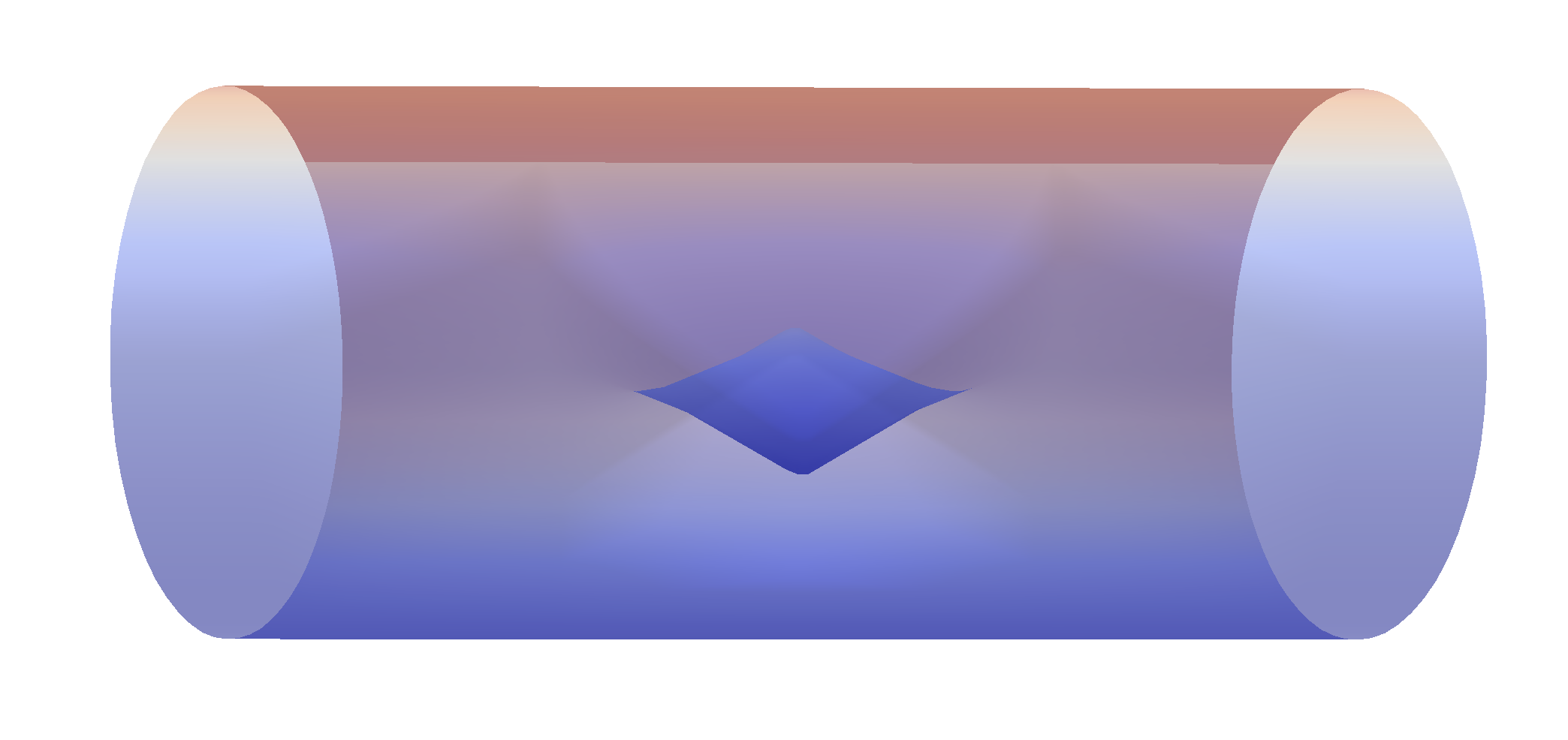}

\subsection*{Acknowledgments}
This material is based upon work supported by, or in part by, the Army Research Laboratory and the Army Research Office under contract/grant number W911NF-18-1-0307.

\printbibliography

%%%%%%%%%%%%%%%%%%%%%%%%%%%%%%%%%

\end{document}